\documentclass[a4paper, 11pt, reqno, twoside]{amsart}
\numberwithin{equation}{section}

\usepackage[T1]{fontenc}
\usepackage[latin1]{inputenc}

\usepackage{graphicx}
\usepackage{epstopdf}

\usepackage{color}
\usepackage[font=scriptsize,labelfont=bf]{caption}

\usepackage{amssymb,amscd,amsmath} 
\usepackage{pdfsync}

\newcommand{\LL}{\mathcal{L}}

\newcommand{\bigOh}{\mathcal O}
\newcommand{\N}{\mathbb N}
\newcommand{\C}{\mathbb C}
\newcommand{\R}{\mathbb R}
\newcommand{\s}{\mathbb S}

\newcommand{\Z}{\mathbb Z}
\newcommand{\Schwartz}{\mathcal{S}}

\newcommand{\diff}{{\mathrm d}}

\DeclareMathOperator{\dist}{dist}
\DeclareMathOperator{\Diff}{D}

\DeclareMathOperator{\F}{{\mathcal F}}

\DeclareMathOperator{\sign}{sgn}
\DeclareMathOperator{\supp}{supp}
\DeclareMathOperator{\im}{Im}
\DeclareMathOperator{\re}{Re}

\newtheorem{theorem}{Theorem}[section]
\newtheorem{lemma}[theorem]{Lemma}
\newtheorem{corollary}[theorem]{Corollary}
\newtheorem{proposition}[theorem]{Proposition}

\newtheorem*{acknowledgments}{Acknowledgments}

\theoremstyle{definition}
\newtheorem{remark}[theorem]{Remark}
\newtheorem{definition}[theorem]{Definition}

\title[On Whitham's conjecture]{\small On Whitham's conjecture of a highest cusped wave for a nonlocal dispersive equation}
\thanks{ME was supported by grant nos.~231668 and 250070 from the Research Council of Norway; EW by grant nos.~621-2012-3753 and 2016-04999 from the Swedish Research Council.}

\author{Mats Ehrnstr\"om}
\address{Department of Mathematical Sciences, Norwegian University of Science and Technology, 7491 Trondheim, Norway}
\email{mats.ehrnstrom@math.ntnu.no}

\author{Erik Wahl\'en}
\address{Centre for Mathematical Sciences, Lund University, PO Box 118, 221\,00 Lund, Sweden\\}
\email{erik.wahlen@math.lu.se}

\renewcommand{\SS}{\mathcal{S}}

\begin{document}

\maketitle

\begin{abstract}
 We consider the Whitham equation \(u_t + 2u u_x+Lu_x = 0\), where \(L\) is the nonlocal Fourier multiplier operator given by the symbol \(m(\xi) = \sqrt{\tanh \xi /\xi}\). G. B. Whitham conjectured that for this equation there would be a highest, cusped, travelling-wave solution. We find this wave as a limiting case at the end of the main bifurcation curve of \(P\)-periodic solutions, and give several qualitative properties of it, including its optimal \(C^{1/2}\)-regularity. An essential part of the proof consists in an analysis of the integral kernel corresponding to the symbol \(m(\xi)\), and a following study of the highest wave. In particular, we show that the integral kernel corresponding to the symbol \(m(\xi)\) is completely monotone, and provide an explicit representation formula for it.  Our methods may be generalised.
\end{abstract}

\section{Introduction}
\noindent In 1967, G.B. Whitham  proposed in  \cite{0163.21104} a nonlocal shallow water wave model for capturing the balance between linear dispersion and nonlinear effects, so that one would have smooth periodic  and solitary waves, but also the features of wave breaking and surface singularities. To accomplish that he considered the symbol 
\[
m(\xi) = \sqrt{{\textstyle \frac{\tanh{\xi}}{\xi} }},
\]
arising as the full frequency dispersion for linear gravity water waves on finite depth, and its inverse Fourier transform,
\begin{equation}\label{eq:K}
K(x) = \frac{1}{2\pi}\int_\R m(\xi) \exp(ix\xi)\,\diff\xi.
\end{equation}
If one denotes by  \(L \colon f \mapsto K \ast f\)
the action by convolution with the kernel \(K\), the \emph{Whitham equation} is the nonlinear, nonlocal evolution equation
\begin{equation}\label{eq:whitham}
u_t + (Lu  + u^2)_x = 0.
\end{equation}
While many  shallow water-wave equations can be written in this form, their symbols are generally leading order approximations of the exact linear dispersion \(m(\xi)\), and therefore behaves radically different for large frequencies \(\xi\); a typical example is the Korteweg--de~Vries equation, whose symbol \(1 - \frac{1}{6}\xi^2\) consists of the two first terms in the Maclaurin series for \(m(\xi)\).  The goal of introducing the operator \(L\), on the other hand, was to weaken the dispersion so as to allow also for solutions with singularities. 

As it turns out, Whitham was correct: the equation \eqref{eq:whitham} features  solitary waves, wave breaking, and, as we will show, periodic waves with a sharp crest. This was not clear. In fact, the operator \(L\) is not only weaker, but much weaker than that of both the KdV equation and almost any other shallow water wave equation, so that the existence of solitary waves was, until recently \cite{EGW11}, an open problem. So was wave breaking (see just below), and the existence of a highest, cusped, travelling wave.

Singularities in solutions of \eqref{eq:whitham} appear in  at least two forms: in the form of wave breaking when the spatial derivative of a bounded solution blows up in the evolution problem, and in the form of a  sharp crests for a travelling wave. Although the idea behind wave breaking was introduced already by Seliger \cite{0159.28502}, the full details for the Whitham equation were settled much later, with \cite{0802.35002}, \cite{MR1668586} and \cite{Hur15}. We, however, shall be concerned with steady waves.

In steady variables \(\varphi(\tilde x) = u(x-\mu t)\) the Whitham equation takes the form
\begin{equation}\label{eq:steadywhitham}
-\mu \varphi + L\varphi +\varphi^2 = 0,
\end{equation} 
where the equation has been integrated once, and the constant of integration set to zero. There is no loss of generality in doing so, since the Galilean change of variables
\[
\varphi \mapsto \varphi + \gamma, \qquad \mu \mapsto \mu + 2 \gamma, \qquad \lambda \mapsto \lambda + \gamma(1-\mu-\gamma),
\] 
maps solutions of \(-\mu \varphi + L\varphi +\varphi^2 = \lambda\) to solutions of a new equation of the same form. The equation \eqref{eq:steadywhitham} can be rigorously justified as a model for shallow water waves travelling rightward with a permanent form and a constant, nondimensionalised, wave speed \(\mu\) \cite{arXiv170509744}, and may also be obtained from the Euler equations via an exponential scaling \cite{MR3390078}. We shall deal with \eqref{eq:steadywhitham} somewhat generally.  \emph{With a solution of the steady Whitham equation we denote a real-valued, continuous and bounded function \(\varphi\) that satisfies \eqref{eq:steadywhitham} almost everywhere.} As Whitham himself conjectured in \cite[p. 479]{MR1699025} (here,  the notation has been changed to match that of~\eqref{eq:steadywhitham}), 
\begin{quote}
\ldots it seems reasonable to assume that in fact a critical height is reached when \(\varphi = \frac{\mu}{2}\). If \(K(x)\) behaves like \(|x|^p\) as \(x \to 0\) and \(\varphi(x)\) behaves like \(\frac{\mu}{2} - |x|^q\), a local argument in \eqref{eq:steadywhitham} suggests that \(2q - 1 = p + q\); hence \( q = p +1\). According to this, the crest would be cusped with \(\varphi \sim \frac{\mu}{2} - |x|^{1/2}\) for \(K\). \footnote{The unit constant in front of \(|x|^{1/2}\) seems to be a computational mistake, cf. \eqref{eq:conjecture}.}
\end{quote}
The simplicity in Whitham's formal argument is striking, even the more so as the equation easily eludes any first attempts at obtaining such a cusped, highest, wave. 

Even though the kernel $K$ in \eqref{eq:K} is real, even, and smooth for all $x \neq 0$ with derivates of rapid decay, as made precise in Proposition~\ref{prop:decay}, Proposition~\ref{prop:kernel decomposition} and Corollary~\ref{cor:improved asymptotics}, it is also singular at the origin, causing nontrivial problems when one wants to analyse it.\footnote{Whitham later approximated the exact kernel \(K\) with a continuous exponential function, resulting in a different equation, known as the Burgers--Poisson equation. That equation has a stronger dispersion than \eqref{eq:whitham}.} We  approach \(K\)  by investigating the signs of its derivatives, taking a route via complex analysis and the theory of completely monotone functions. As it turns out, the kernel \(K\) can be understood via both the theory of Stieltjes functions and the theory of positive definite functions, depending on whether one considers the Laplace or the Fourier transform, respectively. As a by-product of our study we obtain a closed formula for the kernel \(K\) in physical space, as well as for its periodisation. It is worth noting that \(K\)  appears in the classical water-wave problem, as well as in the derivation of numerous  dispersive equations \cite{MR3060183}, so that our analysis  will be useful in these settings as well. 

Building on the results for the integral kernel \(K\) we are able to prove the main result of this paper: the existence of a highest, cusped and periodic travelling-wave solution of \eqref{eq:steadywhitham}, monotonically increasing and smooth between its sole trough and crest in a half-period, and belonging to the H\"older space \(C^{1/2}(\R)\) --- but to no smaller space in the same scale. The proof  thereof has two main components. The construction of a global, locally analytic, curve of sinusoidal, periodic smooth waves along which \(\max\varphi \to \mu/2\) on the one hand, and a detailed analysis of solutions satisfying \(\max \varphi = \mu/2\) on the other. The first part is attained via analytic global bifurcation theory developed by Buffoni, Dancer and Toland \cite{MR1956130}, where we rule out all alternatives along the main bifurcation curve but \(\max\varphi \to \mu/2\), including in particular that the curve could return to a line of constant, but nonzero, solutions (see Figure~\ref{fig:bifurcation} on p.~\pageref{fig:bifurcation} for a qualitative picture of the bifurcation diagram as a whole). It is then straightforward to find a subsequence of waves converging to a solution with \(\max \varphi = \mu/2\), and we use elliptic properties of the equation to rule out  the possibility of the wave speed \(\mu\) vanishing in the limit. 

\begin{figure}
\begin{center}
\includegraphics[width=0.5\linewidth]{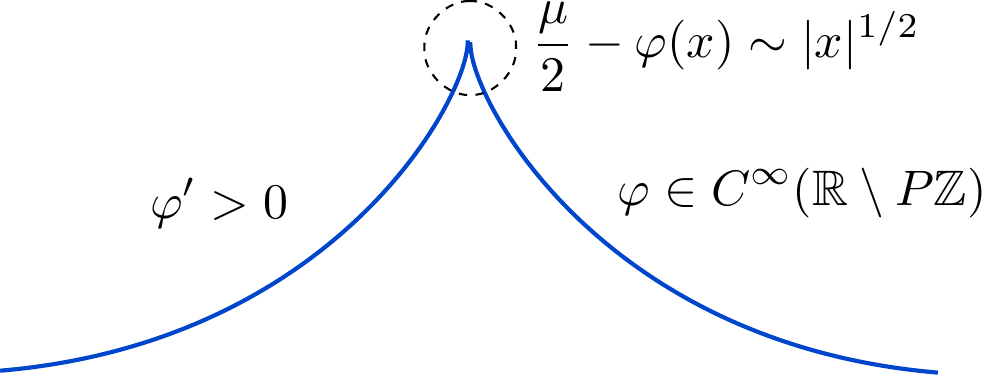}
\caption{The highest wave found in Theorem~\ref{thm:main}. This travelling-wave solution of \eqref{eq:steadywhitham} is obtained as a limit along the main global bifurcation curve established in Theorem~\ref{thm:global_whit}. By construction, the solution is \(P\)-periodic, even, and strictly increasing on the interval \((-P/2,0)\), satisfying \(\varphi(0) = \frac{\mu}{2}\).  As proved in Theorem~\ref{thm:regularity II}, it is furthermore smooth  away from any crest, and obtains its optimal H\"older regularity \(C^{1/2}(\R)\) exactly at the crest. }
\label{fig:highest}
\end{center}
\end{figure}

For an in-depth analysis of the resulting limiting wave a detailed study of the integral equation appears to be unavoidable. Functional-analytic arguments provide us with \(C^\alpha\)-regularity for any \(\alpha < 1/2\), but  not better. To improve our estimates we use several differing  ways of expressing \(\varphi(x+h) - \varphi(x-h)\), which makes it possible to move first- and second-order differences between \(K\) and \(\varphi\) in the integrals that appear. We first move second-order differences to \(K\) and use the \(C^\alpha\)-regularity of \(\varphi\), \(\alpha < 1/2\), to get \(C^{1/2}\)-regularity exactly at the crest. We then place one first-order difference on \(\varphi\) and one on \(K\) to deploy an interpolation argument between the global \(C^{\alpha}\)-regularity and the  \(C^{1/2}\)-regularity exactly at the top, to obtain global \(C^{1/2}(\R)\)-regularity. The highest wave is qualitatively depicted in Figure~\ref{fig:highest}. We conjecture that it is everywhere convex and satisfies 
\begin{equation}\label{eq:conjecture}
\varphi = {\textstyle\frac{\mu}{2}} - {\textstyle \sqrt{\frac{\pi}{8}}}|x|^{1/2} + o(x) \quad\text{ as }\quad x \to 0, 
\end{equation}
but a proof of these facts has so far evaded us. If such a formula holds, then one can show that the constant in front of \(|x|^{1/2}\) must indeed take the value \(\sqrt{\pi /8}\).

Some comments on related recent work on the Whitham equation not mentioned above. The equation \eqref{eq:whitham} features the same kind of Benjamin--Feir instability as the full Euler equations \cite{MR3298879,MR3226084}, although its uni-directional character excludes other (small-amplitude, high-frequency) instabilities seen in the Euler equations \cite{DT15}. It is locally well-posed in \(H^{3/2+}\), in both the periodic setting and on the line \cite{MR3375167}, but a large-time existence theory is so far lacking for equations with a generic nonlinearity and such weak dispersion, see \cite{LannesSaut13} and \cite{MR3188389}. As described above, waves with sufficiently large inclination will eventually break, and numerical data indicates that the form of breaking waves mimics that of the highest wave constructed in this paper \cite{MR3317254}. The results presented in this paper are in turn based on \cite{EK08} and \cite{EK11}, in which global branches of periodic, but smooth, periodic solutions were analytically constructed and numerically investigated. 

The outline of our investigation is as follows. In Section~\ref{sec:K} we inspect the integral kernel \(K\) corresponding to the symbol \(m(\xi)\). Although some of our results are valid for general completely monotone functions, the  results with most consequence for our further investigation  are Propositions~\ref{prop:g is Stieltjes} and~\ref{prop:Whitham formula}, where we prove that \(K\) is completely monotone---meaning that all its odd derivatives are negative on a half-line, and contrariwise for the even derivatives---and give a closed formula for it. In Section~\ref{sec:Kp} we continue the study of the integral operator \(L\), now for the periodised integral kernel \(K_P\), and give some useful properties of  \(L\) in general. Interestingly,  \(K_P\) is completely monotone as well, on a half-period (this is not a coincidence, but a general fact for integrable completely monotone functions). A closed formula for the periodised kernel is given in Corollary~\ref{cor:periodic formula}. 

In Section~\ref{sec:nodal} we prove some general lemmas about solutions of \eqref{eq:steadywhitham}, whereof the most important to us is Theorem~\ref{thm:nodal}, which establishes the nodal properties of solutions along the main bifurcation branch to be constructed. As it turns out, \eqref{eq:steadywhitham} satisfies a  maximum principle (touching lemma), making it resemblant of  an elliptic equation. The nodal properties are essential in avoiding the closed-loop alternative in the global bifurcation analysis,  but they also give information about the waves in their own right. 

Section~\ref{sec:singularity} is the main part of the paper, in the sense that it contains the a priori analysis the highest wave. It is also the most technical, making use of both Besov spaces and, mostly, of integral estimates adapted for the assumed optimal regularity of the wave. Since \(K(x) \sim |x|^{-1/2}\) and we expect \(\frac{\mu}{2} - \varphi \sim |x|^{1/2}\), both relations for small values of \(x\), one difficulty is that the integral \(\int K^\prime(y) (\frac{\mu}{2} - \varphi(y))\,\diff y\) diverges exactly at the expected regularity; another is that the point  where \(\varphi = \frac{\mu}{2}\) must be treated separately from other points. The main results of Section~\ref{sec:singularity} are summarised in Theorem~\ref{thm:regularity II} about the regularity of the highest wave. Lastly, we revisit in Section~\ref{sec:global}  the bifurcation analysis from \cite{EK08, EK11}, ultimately proving that there is a sequence of waves converging to a wave of greatest height \(\varphi(0) = \frac{\mu}{2}\), with a nontrivial wave speed \(\mu \in (0,1)\). We underscore that several parts of Section~\ref{sec:global} are new with respect to \cite{EK08, EK11}, including the bifurcation formulas given in the proof of Theorem~\ref{thm:local_whit}. The methods developed in this paper may be generalised to other dispersive equations. 

Finally, the existence of a highest, cusped travelling-wave solution of the Whitham equation was announced earlier in \cite{E15_highest} (without proofs). The paper at hand provides a complete account of this fact, as well as several improvements---the most eminent examples being the regularity of the highest wave, and the properties of the kernels \(K\) and \(K_P\).

\section{Completely monotone functions and the integral kernel \(K\)} \label{sec:K}
In this section we investigate the properties of the integral kernel \(K\) in \eqref{eq:K}. Two routes towards understanding this transform are described---via positive definite functions (related to the Fourier transform), and via completely monotone functions (related to the Laplace transform). We start our exposition with a survey and analysis of completely monotone functions in general, whereafter the applications to the Whitham symbol \(m\) are investigated. Among other things, we obtain complete monotonicity and an explicit series expression for the Whitham kernel \(K\). 
 
\subsection*{Regularity properties}
Recall \eqref{eq:K}. We consider in this paper the Fourier transform as a continuous isomorphism  $\F \colon \Schwartz^\prime(\R) \to \Schwartz^\prime(\R)$ on the space \(\Schwartz^\prime(\R)\) of tempered distributions, defined by duality from the Fourier transform on the Schwartz space \(\Schwartz(\R)\) of smooth and rapidly decaying functions. Our normalisation of \(\F\) is
\[
(\F f)(\xi) = \int_{\R} f(x) \exp(-ix\xi) \,\diff x \quad\text{ for }\quad f \in \Schwartz(\R),
\]
which implies that \((\F^{-1} f)(x)=\frac1{2\pi} (\F f)(-x)\).
Clearly \(m\in \Schwartz^\prime(\R)\), whence \(K\) exists at least as an element of \(\Schwartz^\prime(\R)\). 
However, since \(m\) is smooth and all of its derivatives are integrable, 
$K$ is actually smooth for $x\ne 0$, and all its derivatives have rapid decay.
In fact, since $m$ is analytic in a strip containing the real axis, $K$ and all of its derivatives are exponentially decaying.

\begin{proposition}\label{prop:decay}
For any fixed $s_0 \in (0, \pi/2)$, $n\ge 0$, one has
\[
|\Diff_x^{n} K(x)| \lesssim \exp(-s_0 |x|)
\]
for all $|x|\ge 1$.
\end{proposition}

\begin{remark}
Throughout this paper, \(\lesssim\) and \(\gtrsim\) shall indicate inequalities that hold up to a uniform positive factor. When the factor involved depends on some additional parameter or function, this will be indicated with subscripts such as \(\gtrsim_\mu\).
\end{remark}

\begin{remark}
More precise asymptotics for \(K(x)\) as \(|x|\to \infty\) is given in Corollary \ref{cor:improved asymptotics} below.
\end{remark}

\begin{proof}
Since $K$ is even it suffices to consider $x\ge 1$.
Note that the integral 
\[
\int_\R \exp(ix\xi) m(\xi)\, \diff\xi
\]
converges conditionally.
Indeed,
\begin{align*}
\int_{-R}^R \exp(ix\xi) m(\xi)\, \diff\xi&=2\int_{0}^R \cos(x\xi) m(\xi)\, \diff\xi\\
&=\frac{2\sin(Rx)m(R)}{x} -\frac{2}{x} \int_0^R \sin(x\xi) m^\prime(\xi)\, \diff\xi\\
&\to -\frac{2}{x} \int_0^\infty \sin(x\xi) m^\prime(\xi)\, \diff\xi
\end{align*}
as $R\to \infty$; 
the latter integral converges absolutely since $m^\prime(\xi)=\mathcal{O}(|\xi|^{-3/2})$ as $|\xi|\to  \infty$. 
The function \(\zeta \mapsto m(\zeta)\) is analytic in \(\C\setminus S\), where $S=\cup_{k=1}^\infty i[k\pi-\pi/2, k\pi]\cup i[-k\pi, -k\pi+\pi/2]$.
Furthermore,
\begin{equation}
\label{eq:tanh bound}
\begin{aligned}
|\tanh(\zeta)|^2
&=\left|\frac{\exp(\xi+is)-\exp(-\xi-is)}{\exp(\xi+is)+\exp(-\xi-is)}\right|^2\\
&=\frac{\exp(2\xi)-2\cos(2s)+\exp(-2\xi)}{\exp(2\xi)+2\cos(2s)+\exp(-2\xi)}\\
&\le \coth^2 \xi \\
&\le \coth^2 \xi_0
\end{aligned}
\end{equation}
when $|\xi|\ge \xi_0>0$, in which $\zeta=\xi+is$. 
Noting that  $\exp(ix\zeta)$ is bounded for \(x > 0\) when $\im \zeta \ge 0$, 
we therefore obtain that
\begin{equation}
\label{eq:deformation estimate}
\lim_{|\xi|\to \infty} \sup_{s \ge 0} |\exp(ix\zeta)m(\xi+is)|=0.
\end{equation}
Fix a number $s_0\in (0, \pi/2)$.
Using Cauchy's theorem on a bounded rectangle with vertices $\pm R$, $\pm R+is_0$, and letting $R\to \infty$, it follows that
\begin{equation}
\begin{aligned}
\label{eq:first deformation}
\int_\R \exp(ix\xi) m(\xi)\, \diff\xi&= 
\int_\R \exp(ix(\xi+is_0))m(\xi+is_0)\, \diff\xi\\
&
=\exp(-xs_0) \int_\R \exp(ix\xi)m(\xi+is_0)\, \diff\xi.
\end{aligned}
\end{equation}
Integrating by parts and using the estimate
\(|\partial_\xi m(\xi+is_0)|=\mathcal{O}(|\xi|^{-3/2})\) as $|\xi|\to \infty$, we obtain the desired exponential decay of $K$.

In order to estimate the derivatives of $K$, we note that $|\F (x^{n} \Diff_x^n K)(\xi)|= |\Diff_\xi^{n} (\xi^n m(\xi))|$,  for any \(n\ge 0\), where $\Diff_\xi^{n} (\xi^n m(\xi))$ extends analytically to the strip $0\le \im \zeta <\pi/2$ and satisfies the estimate \(|D_\xi^{n+k} ((\xi+is_0)^n m(\xi+is_0))|=\mathcal{O}(|\xi|^{-1/2-k})\) as $|\xi|\to \infty$, for any $k\ge 0$ and $s_0\in (0, \pi/2)$.
Repeating the above argument, we obtain exponential decay for $x^n \Diff_x^n K(x)$ and hence also for $\Diff_x^n K$.
\end{proof}

Due to the fact that \(m\not \in L^1(\R)\), it follows that \(K\) is singular at the origin. We can give a precise description of this singularity as follows.

\begin{proposition}
\label{prop:kernel decomposition}
The Whitham kernel satisfies
\[
K(x)=\frac{1}{\sqrt{2\pi|x|}}+K_\text{reg}(x),
\]
where $K_\text{reg}$ is real analytic on $\R$. 
\end{proposition}

\begin{proof}
Write
\[
\sqrt{\frac{\tanh \xi}{\xi}}=\frac{1}{\sqrt{|\xi|}}+\frac{\sqrt{\tanh |\xi|}-1}{\sqrt{|\xi|}}.
\]
The first term has inverse Fourier transform $1/\sqrt{2\pi|x|}$, while the second term is integrable and exponentially decaying and hence has a real-analytic transform. 
\end{proof}

\subsection*{Positivity and monotonicity properties: general theory}

Our next aim is to show certain positivity and monotonicity properties of $K$. We begin by proving such results for the Fourier transforms of a general class of functions. In the next subsection, we then show that the Whitham symbol $m$ belongs to this class. Much of the general theory discussed in this section is adapted from the monograph \cite{SchillingSongVondracek12}, although we slightly extend some of it. 
Most importantly, we relate it to the theory of positive definite functions and the kernel \(K\).
\medskip

\begin{definition}
A function $g\colon (0, \infty)\to \R$ is  called {\em completely monotone} if it is of class $C^\infty$ and
\begin{equation}
\label{eq:completely monotone}
(-1)^n g^{(n)}(\lambda)\ge 0
\end{equation}
for all \(n \in \N_0\) and all \(\lambda >0\).
\end{definition}

We shall sometimes say that a function is completely monotone on some interval (typically, a half-period), meaning that \eqref{eq:completely monotone} holds on that interval. Moreover, if \(g\colon \R\setminus \{0\} \to \R\) is even, we shall say that \(g\) is completely monotone if it is completely monotone on the interval \((0,\infty)\).
One of the main reasons for introducing completely monotone functions is that they are precisely the functions which arise as Laplace transforms of measures.
This is known as the Bernstein, or Bernstein--Hausdorff--Widder, theorem. We adopt here the convention that \emph{a measure is always countably additive and positive}.

\begin{theorem}[Bernstein]\label{thm:HBW}
Let $g$ be completely monotone. Then it is the Laplace transform of a unique Borel measure $\mu$ on $[0,\infty)$, i.e.
\begin{equation}
\label{eq:representation of completely monotone functions}
g(\lambda)=\LL(\mu; \lambda):=\int_{[0,\infty)} \exp(-\lambda s) \, \diff\mu(s).
\end{equation}
Conversely, if $\mu$ is a Borel measure on $[0,\infty)$ with $\LL(\mu;\lambda)<\infty$ for every $\lambda>0$, then $\lambda \mapsto 
\LL(\mu;\lambda)$ is a completely monotone function.
\end{theorem}
For a proof of this result, see \cite[Theorem 1.4]{SchillingSongVondracek12}. A consequence of Bernstein's theorem is that if $g$ is completely monotone, then 
\eqref{eq:completely monotone} holds with strict inequality for every $\lambda$ and every $n$, unless $g$ is identically constant. Note also that  the measure $\mu$ in \eqref{eq:representation of completely monotone functions} is finite if and only if $\lim_{\lambda\searrow 0} g(\lambda)<\infty$.
\medskip

For later use we introduce the following subclass of the completely monotone functions.

\begin{definition}\label{def:Stieltjes function}
A function $g\colon (0, \infty)\to [0,\infty)$ is  called a (nonnegative) {\em Stieltjes function} if it can be written in the form
\begin{equation}
\label{eq:Stieltjes}
g(\lambda)=\frac{a}{\lambda}+b+\int_{(0,\infty)} \frac{1}{\lambda+t}\, \diff\sigma(t),
\end{equation}
where \(a, b\ge 0\) are constants and $\sigma$ is a  Borel measure on $(0, \infty)$ such that $
\int_{(0,\infty)} \frac{\diff\sigma(t)}{1+t}<\infty$.
\end{definition}

Note that if \(g\) has a finite limit at the origin, then \(a=0\) and \(\int_{(0,\infty)} \frac{\diff\sigma(t)}{t}<\infty\) by Fatou's lemma.
Moreover, \(b=\lim_{\lambda \to \infty} g(\lambda)\).
The fact that Stieltjes functions are completely montone is proved in \cite{SchillingSongVondracek12}.

\begin{theorem}{\cite[Theorem 2.2]{SchillingSongVondracek12}}
Stieltjes functions are completely monotone. A completely monotone function is a Stieltjes function if and only if 
the measure \(\mu\) in \eqref{eq:representation of completely monotone functions} is absolutely continuous on \((0,\infty)\) and its Radon--Nikodym derivative is completely monotone.
\end{theorem}

It turns out that any Stieltjes function has an analytic extension to the cut complex plane $\C\setminus (-\infty,0]$. This property gives a complete characterisation of the class of Stieltjes functions.
Let $\C_+=\{z\in \C\colon \im z>0\}$ and $\C_-=\{z\in \C\colon \im z<0\}$. 

\begin{theorem}{\cite[Corollary 7.4]{SchillingSongVondracek12}}
\label{thm:Stieltjes extensions}
Let $g$ be a positive function on $(0,\infty)$. Then $g$ is a Stieltjes function if and only if the limit 
$\lim_{\lambda \searrow 0} g(\lambda)$  exists in $[0,\infty]$ and $g$ extends analytically to  $\C\setminus (-\infty, 0]$ such that 
$\im z \cdot \im g(z)\le 0$.
\end{theorem}

\begin{remark}
Note that positive constant functions are examples of Stieltjes functions. It follows easily by basic properties of analytic functions that a nonconstant Stieltjes function maps $\C_+$ to $\C_-$. 
Note also that if $g$ is not identically $0$, then $1/g(z)$ is a Nevanlinna function (also known as Herglotz or Pick functions).
The corresponding function $1/g(\lambda)$ is then a complete Bernstein function, see \cite{SchillingSongVondracek12}.
\end{remark}

It is possible to compute the measure $\sigma$ in \eqref{eq:Stieltjes} using the analytic extension of $g$. The following result follows from \cite[Corollary 6.3]{SchillingSongVondracek12} and the fact that  $\lambda g(\lambda)$ is a complete Bernstein function if $g(\lambda)$ is a Stieltjes function (see \cite[Theorem 6.2]{SchillingSongVondracek12}).

\begin{theorem}
The measure $\sigma$ in \eqref{eq:Stieltjes} can be recovered from $g$ by the formula
\begin{equation}
\label{eq:inversion formula}
\sigma(u,v]=-\lim_{\delta \searrow 0} \lim_{h\searrow 0} \frac{1}{\pi} \int_{u+\delta}^{v+\delta} \im g(-t+ih) \, \diff t, \quad 0<u<v<\infty.
\end{equation}
\end{theorem}

We also record the following lemma which follows easily from Theorem \ref{thm:Stieltjes extensions}.

\begin{lemma}
\label{lemma:composition}
If $g$ is a Stieltjes function, then so is $g^\alpha$ for any $\alpha \in (0,1]$.
\end{lemma}

Next, we are interested in characterising functions with a positive Fourier transform. We refer to \cite{MR2284176} for the following standard results.

\begin{definition}
A function \(f\colon \R^d \to \C\) is said to be \emph{positive definite} if for every \(n \in \N\), the \(n \times n\)-matrix with values  \(a_{ij} = f(\xi_i - \xi_j)\), \(1 \leq i,j \leq n\), is positive semi-definite. 
\end{definition}

\begin{theorem}[Schur]\label{thm:schur}
Let \(\{f_j\}_j\) be a countable family of positive definite functions. Then \(\prod_{j} f_j\) is positive definite.
\end{theorem}

\begin{theorem}[Bochner]\label{thm:bochner}
Any positive definite function continuous at zero is the Fourier transform of a finite Borel measure.
\end{theorem}

\begin{remark}
A common way of stating  Bochner's theorem is as a one-to-one correspondence between continuous positive definite functions and probability measures. This form of the statement is in agreement with Theorem~\ref{thm:bochner} as long as one requires \(f(0) = 1\) for the positive definite functions.
\end{remark}

Let \(f\) be a positive definite function. If \(f^r\) is positive definite for any real power \(r \geq 0\), then \(f\) is said to be \emph{infinitely divisible}.   By definition, any root \(f^{1/n}\) of an infinitely divisible function \(f\) is a positive definite function. It follows from Theorem~\ref{thm:schur} that any product of infinitely definite functions is again infinitely divisible. Moreover, \(f\) is infinitely divisible if and only if \(f^{1/n}\) is positive definite for any \(n \in \N\), since products and pointwise limits of positive definite functions are positive definite.
\medskip

We next recall Schoenberg's theorem which links completely monotone functions and positive definite functions.

\begin{theorem}\cite{MR1503439}\label{thm:characterisation}
A function \(g \colon [0,\infty)\to \R\) continuous at zero is completely monotone if and only if \(g(|\cdot|^2)\) is positive definite 
 on \(\R^d\) for all \(d \in \N\). 
\end{theorem}

We have the following two results, giving us properties of transforms of completely monotone and Stieltjes functions, respectively.

\begin{proposition}\label{prop:completely monotone Fourier}
Let \(f\colon \R \to \R\) and \(g\colon [0,\infty) \to \R\) be two functions satisfying \(f(\xi)=g(\xi^2)\).
Then \(f\) is  the Fourier transform of an even, integrable function such that
\((\F^{-1} f)(\sqrt{\cdot})\) is completely monotone if and only if \(g\) is completely monotone with \(\lim_{\lambda \searrow 0} g(\lambda)<\infty\) and \(\lim_{\lambda \to \infty} g(\lambda)=0\). In this case,  \(\F^{-1} f\) is smooth and monotone outside of the origin.
\end{proposition}

\begin{remark}
For a related result by Bochner (on subordinate Brownian motions), see \cite[Theorem 4.3.3]{MR0072370} and \cite[Example 13.16]{SchillingSongVondracek12}.
\end{remark}

\begin{proof}
When \(g\) is completely monotone and continuous at zero,  Bochner's theorem guarantees us that \(f(\xi)=g(\xi^2)\) is the Fourier transform of a measure.  
In fact, by Bernstein's theorem (cf.~Theorem~\ref{thm:HBW}), we have that
\[
f(\xi)=\mu(\{0\})+\int_{(0,\infty)} \exp(-\xi^2 t)\, \diff\mu(t)
\]
for some finite Borel measure $\mu$ on $[0, \infty)$. 
Since $\lim_{\xi\to \infty} f(\xi)=0$, we infer that $\mu(\{0\})=0$, whence
\begin{equation}\label{eq:Bochner_formula}
f(\xi)=\int_{(0,\infty)} \exp(-\xi^2 t)\, \diff\mu(t).
\end{equation}
Noting that
\[
\exp(-t \xi^2)=\F\left(\frac{\exp(-\frac{(\cdot)^2}{4 t })}{\sqrt{4\pi  t }}\right)(\xi)=
 \int_{\R} \frac{\exp(-\frac{x^2}{4t})}{\sqrt{4\pi t}}\exp(-i x\xi) \, \diff x, 
\]
it follows that
\begin{equation}
\label{eq:f Fubini}
\begin{aligned}
f(\xi)&= \int_{(0,\infty)} \left( \int_{\R} \frac{\exp(-\frac{x^2}{4t})}{\sqrt{4\pi t }} \exp(-i x\xi)\, \diff x \right) \, \diff\mu(t)\\
&= \int_{\R} \left(\int_{(0,\infty)}  \frac{\exp(-\frac{x^2}{4 t })}{\sqrt{4\pi t }} \, \diff\mu( t )\right) \exp(-i x\xi)\, \diff x,
\end{aligned}
\end{equation}
where we have used Fubini's theorem; in order to verify that it applies one can consider the change of variables \(y=x/\sqrt{t}\), recalling that \(\mu\) is finite. Hence,
\begin{equation}
\label{eq:inverse Fourier Bernstein I}
(\F^{-1} f)(x) =  \int_{(0,\infty)} \frac{\exp(-\frac{x^2}{4 t })}{\sqrt{4\pi t }} \, \diff\mu(t).
\end{equation}
This calculation together with Bochner's theorem implies that \( \F f \in L^1(\R)\cap C^\infty(\R\setminus\{0\})\), so that \(\F f\) is actually given by a function and not just a measure (in this and the following proof it does not matter if we consider \(\F f\) or \(\F^{-1}f\), since they differ only by a constant factor). The positivity of \(\F f\) is clear from the above formula, too, and we in addition see that 
\( \F f\) is monotone for $x>0$. More precisely, 
\[
(\F^{-1} f)^\prime(x)=-\frac{1}{4\sqrt{\pi}} \int_{(0,\infty)} \frac{x\exp(-\frac{x^2}{4t})}{t^{3/2}}  \,\diff\mu(t) <0.
\]
Finally, the fact that \((\F^{-1} f)(\sqrt{\cdot})\) is completely monotone is a consequence of Bernstein's theorem and the computation
\begin{equation}
\label{eq:inverse Fourier Bernstein II}
(\F^{-1} f)(\sqrt{\lambda}) =  \int_{(0,\infty)} \frac{\exp(-\frac{\lambda}{4 t })}{\sqrt{4\pi t }} \, \diff\mu(t)=
\int_{(0,\infty)} \exp(-s\lambda)\, \diff\tilde \mu(s),
\end{equation}
in which the measure \(\tilde \mu\) is given by
\begin{equation}
\label{eq:tilde mu definition}
\diff \tilde \mu=\sqrt{\frac{\cdot}{\pi}} \diff (\psi_*(\mu)),
\end{equation}
where \(\psi(t)=\frac{1}{4t}\) and \(\psi_*(\mu)\) is the pushforward of \(\mu\) by \(\psi\).

Conversely, suppose that \(f\) is the Fourier transform of an even, integrable function \(\F^{-1} f\) and that
\((\F^{-1} f)(\sqrt{\cdot})\) is completely monotone. Then we can write \((\F^{-1} f)(\sqrt{\cdot})\) in the form
\[
(\F^{-1} f)(\sqrt{\lambda}) =
\int_{[0,\infty)} \exp(-s\lambda)\, \diff\tilde \mu(s),
\]
where \(\tilde \mu\) is obtained using Bernstein's theorem. Consequently,
\[
(\F^{-1} f)(x) =
\int_{[0,\infty)} \exp(-sx^2)\, \diff\tilde \mu(s),
\]
and integrating this relation using the change of variables \(y=\sqrt{s}x\) yields that \(\diff \tilde \mu(s)/\sqrt{s}\) is a finite measure. In particular, \(\tilde \mu\) has no mass at \(0\), so that \eqref{eq:inverse Fourier Bernstein II} holds, 
where \(\mu\) is the finite measure defined by \eqref{eq:tilde mu definition}. Consequently, we have \eqref{eq:inverse Fourier Bernstein I} and the calculation \eqref{eq:f Fubini} is now justified by Fubini's theorem.
Thus, \eqref{eq:Bochner_formula} holds and this in turn implies that \(g\) is completely monotone with 
\(\lim_{\lambda \searrow 0} g(\lambda)=0\) and \(\lim_{\lambda \to \infty} g(\lambda)=0\).
\end{proof}

When the function \(g\) is Stieltjes one can sharpen these conclusions. It is already clear from Theorem~\ref{thm:characterisation} and Lemma~\ref{lemma:composition} that $g(\xi^2)$ is infinitely divisible. In addition, we have:

\begin{proposition}
\label{prop:Stieltjes Fourier}
Let \(f\) and \(g\) be two functions satisfying \(f(\xi)=g(\xi^2)\). Then \(f\) is the Fourier transform of an even, integrable and completely monotone  function if and only if \(g\) is Stieltjes with  \(\lim_{\lambda \searrow 0} g(\lambda)<\infty\) and \linebreak \(\lim_{\lambda \to \infty} g(\lambda)=0\).
One has
\begin{equation}
\label{eq:g representation}
g(\lambda)=\int_{(0,\infty)} \frac{1}{\lambda+t}\, \diff\sigma(t), 
\end{equation}
and, with \(\psi = \sqrt{\cdot}\), the pushforward of \(\sigma\) by \(\psi\) relates \(f\) to \(g\) via
\begin{equation}\label{eq:f representation}
(\F^{-1} f)(x)=\int_{(0,\infty)} \exp(-s|x|) \, \diff\mu(s), \qquad \diff \mu=\frac{1}{2(\cdot)}\diff( \psi_*(\sigma)).
\end{equation}
\end{proposition}

\begin{remark}
The main part of this result is known in the theory of subordinate Brownian motions; see \cite[Proposition 2.14]{MR2860308} and \cite[Example 13.16]{SchillingSongVondracek12}. We include a proof for completeness since our result is slightly different as well as phrased in a different language.
\end{remark}

\begin{proof}
Suppose first that \(g\) is a Stieltjes function with \(\lim_{\lambda \searrow 0} g(\lambda)<\infty\) and \(\lim_{\lambda \to \infty} g(\lambda)=0\).
By assumption, 
\(g\) is given by \eqref{eq:g representation}
with \(\int_{(0,\infty)}\frac{\diff\sigma(t)}{t}<\infty\) (see the remark after Definition \ref{def:Stieltjes function}) and hence we obtain from \(f(\xi)=g(\xi^2)\) that
\[
f(\xi) =\int_{(0,\infty)} \frac{1}{\xi^2+t}\, \diff\sigma(t).
\]
From the transform
\begin{equation}\label{eq:exponential transform} 
\F^{-1} \left( \frac{s}{\xi^2+s^2} \right)(x) =\frac12 \exp(-s\, |x|),
\end{equation}
with \(s = \sqrt{t}\), and an argument as in the proof of Proposition \ref{prop:completely monotone Fourier}, we get that
\begin{equation}\label{eq:fhat_Stieltjes}
(\F^{-1} f)(x)= \int_{(0,\infty)} \frac{\exp(-\sqrt{t}\, |x|)}{2\sqrt{t}} \, \diff\sigma(t).
\end{equation}
Here, one makes the changes of variables
\(y = \sqrt{t} x\)
in order to justify the use of Fubini's theorem. 
Making the change of variables \(t \mapsto \sqrt{t} = \psi(t)\) in the integral, we obtain \eqref{eq:f representation} in the variable \(s = \sqrt{t}\). In particular,  \(\F f\) is completely monotone by Bernstein's theorem (\(\diff \mu(s)/s\) is finite). The evenness of  \(\F f\) follows immediately from the evenness of \(g(\xi^2)\), and the fact that  \(\F f \in L^1(\R)\) is a consequence of Proposition \ref{prop:completely monotone Fourier}.

Conversely, suppose that \(\F f\in L^1(\R)\) is even and completely monotone.
Then
\[
(\F^{-1} f)(x)=\int_{[0, \infty)} \exp(-|x|s)\,\diff\mu(s),
\]
for some Borel measure \(\mu\) on \([0,\infty)\) and the integrability of  \(\F f\) implies that \(\diff \mu(s)/s\) is a finite measure. In particular, \(\mu\) has no mass at \(0\) so that the left-most equality in \eqref{eq:f representation} holds.
Thus, we have \eqref{eq:fhat_Stieltjes}  with   \(\sigma\) defined by \eqref{eq:f representation}, and
\eqref{eq:exponential transform} together with Fubini's theorem yield \eqref{eq:g representation}.
Moreover, it is easily seen that \(\diff \sigma(t)/t\) is finite so that \(g\) is a Stieltjes function with the desired properties.
\end{proof}

\subsection*{Positivity and monotonicity properties: the Whitham kernel}

Note that we can write the Whitham symbol as $m(\xi)=g(\xi^2)$, where
\begin{equation}\label{eq:gm}
g(\lambda)=\sqrt{\frac{\tanh \sqrt{\lambda}}{\sqrt{\lambda}}}, \qquad \lambda\ge 0.
\end{equation}

\begin{proposition}\label{prop:g is Stieltjes}
$(g(\lambda))^{2\alpha}$ is a Stieltjes function for any $\alpha \in (0, 1]$.
\end{proposition}

\begin{proof}
To see this, 
note that the reciprocal
\[
\lambda \to  \frac{\sqrt{\lambda}}{\tanh \sqrt{\lambda}}
\]
is positive on $(0, \infty)$ with the finite limit $1$ as $\lambda \searrow 0$, and extends to an analytic function on $\C\setminus (-\infty, 0]$ 
if we let $\sqrt{\lambda}$ denote the principal branch of the square root.
It also maps $\C_+$ to $\C_+$. Indeed, a straightforward calculation shows that
\begin{align*}
\im \left(\frac{z}{\tanh z}\right)&=\frac{2}{|\exp(z)-\exp(-z)|^2} (\im z \sinh(2\re z)-\re z \sin(2\im z))\\
&> 
\frac{4}{|\exp(z)-\exp(-z)|^2} (\im z \re z-\re z \im z)=0
\end{align*}
when $\re z, \im z>0$, from which it follows that \(
\im(\sqrt{\lambda}/\tanh \sqrt{\lambda}))> 0\)
when $\im \lambda >0$. This implies that \(\lambda \mapsto \tanh \sqrt{\lambda}/\sqrt{\lambda}\)
satisfies the conditions of  Theorem \ref{thm:Stieltjes extensions}. 
Hence, $(g(\lambda))^{2\alpha}$ is a Stieltjes function by Lemma \ref{lemma:composition}.
\end{proof}

Throughout the rest of Section~\ref{sec:K} we let $\alpha=1/2$. The following is our main result concerning the kernel \(K\), and will be used repeatedly in the later sections.

\begin{proposition}\label{prop:Whitham formula}
The Whitham kernel can be expressed in the form \eqref{eq:f representation}. The Borel measure $\sigma$ in the same formula satisfies $\int_{(0,\infty)}\frac{\diff\sigma(t)}{t}<\infty$, and is absolutely continuous with density
\[
\frac{1}{\pi} \sum_{n=1}^\infty  \sqrt{\frac{|\tan \sqrt{t}|}{\sqrt{t}}}\chi_{((\frac{(2n-1)\pi}{2})^2, (n\pi)^2)}(t).
\]
Thus
\begin{equation}
\label{eq:explicit expression for Whitham kernel}
K(x)=\frac1{\pi} \sum_{n=1}^\infty\int_{\frac{(2n-1)\pi}{2}}^{n\pi} \exp(-s |x|)\sqrt{\frac{|\tan s|}{s}}\, \diff s,
\end{equation}
and \(K\) is completely monotone on $(0,\infty)$. In particular,  it is positive, strictly decreasing and strictly convex for $x>0$.
\end{proposition}

\begin{proof}
For \(g\) as in \eqref{eq:gm} we have $\lim_{\lambda\to \infty} g(\lambda)=0$ and  $\lim_{\lambda\to 0} g(\lambda)=1$. 
Applying Proposition \ref{prop:Stieltjes Fourier} we immediately obtain the first part of the proposition. The inversion formula \eqref{eq:inversion formula} furthermore gives us
\begin{align*}
\sigma(u,v]&=-\lim_{h\searrow 0} \frac{1}{\pi} \int_{(u,v]} \im 
\sqrt{\frac{\tanh\sqrt{-t+ih}}{\sqrt{-t+ih}}}
\, \diff t\\
&=
-\lim_{h\searrow 0} \frac{1}{\pi} \int_{(u,v]} \im 
\sqrt{\frac{\tan\sqrt{t-ih}}{\sqrt{t-ih}}}
\, \diff t.
\end{align*}
We get a contribution from each interval on which $\tan\sqrt{t}/\sqrt{t}$ is negative, i.e., from each interval
\[
((\pi/2)^2, \pi^2),\quad ((3\pi/2)^2, (2\pi)^2),\quad \ldots\quad,\quad (((2n-1)\pi/2)^2, (n\pi)^2),\quad \ldots,
\]
giving the announced expression for $\sigma$. The formula \eqref{eq:explicit expression for Whitham kernel}   for $K$ then follows by  substituting this expression for \(\sigma\) in
\eqref{eq:f representation} and making the change of variables $s=\sqrt{t}$ in order to determine \(\mu\).
\end{proof}

\begin{remark}
We remark that \eqref{eq:explicit expression for Whitham kernel} could also be obtained by deforming the contour in the calculation of the Fourier transform of $m(\xi)$ further. Assume that $x>0$
and recall from \eqref{eq:first deformation} that
\begin{align*}
\int_\R \exp(ix\xi) m(\xi)\, \diff\xi&= 
\int_\R \exp(ix(\xi+is_0))m(\xi+is_0)\, \diff\xi\\
&=\exp(-xs_0) \int_\R \exp(ix\xi)m(\xi+is_0)\, \diff\xi,
\end{align*}
for any $s_0\in (0, \pi/2)$.
We can extend the contour by replacing $s_0$ with a number  $s_1\in (\pi, 3\pi/2)$, obtaining
\begin{equation}
\label{eq:second deformation}
\begin{aligned}
&\int_\R \exp(ix(\xi+is_0))m(\xi+is_0)\, \diff\xi\\
&= 
\int_\R \exp(ix(\xi+is_1))m(\xi+is_1)\, \diff\xi\\
&\qquad
+i\lim_{{s}\searrow 0} \int_{\frac{\pi}{2}}^{\pi} \left(\exp(ix({s}+is)) m({s}+is) - \exp(ix(-{s}+is)) m(-{s}+is)\right) \, \diff s\\
&=
\int_\R \exp(ix(\xi+is_1))m(\xi+is_1)\, \diff\xi
+2\int_{\frac{\pi}{2}}^{\pi} \exp(-xs)
\sqrt{\frac{|\tan s|}{s}}\, \diff s;
\end{aligned}
\end{equation}
see Figure \ref{fig:contour}.
Repeating this procedure, we may replace $s_1$ by a number $s_N \in (N\pi, (2N+1)\pi/2)$ and obtain
\begin{align*}
\int_\R \exp(ix\xi) m(\xi)\, \diff\xi&= 
\int_\R \exp(ix(\xi+is_N))m(\xi+is_N)\, \diff\xi\\
&\quad
+2 \sum_{n=1}^N\int_{\frac{(2n-1)\pi}{2}}^{n\pi} \exp(-s x) \sqrt{\frac{|\tan s|}{s}}\, \diff s,
\end{align*}
for $N=1,2,3,\ldots$. Taking $s_N=\pi/4+N\pi$, and noting that 
\(|\partial_\xi m(\xi+is_N)|\lesssim (|\xi|+1)^{-3/2}\) uniformly in $N$, we find that
\[
\int_\R \exp(ix(\xi+is_N))m(\xi+is_N)\, \diff\xi=\exp(-xs_N)\int_\R \exp(ix\xi)m(\xi+is_N)\, \diff\xi\to 0
\]
as $N\to \infty$  (the convergence is uniform for $x\ge x_0>0$). It follows that
\[
\frac{1}{2\pi}\int_\R \exp(ix\xi) m(\xi)\, \diff\xi= 
\frac{1}{\pi} \sum_{n=1}^{\infty}\int_{\frac{(2n-1)\pi}{2}}^{n\pi} \exp(-s x) \sqrt{\frac{|\tan s|}{s}}\, \diff s,\quad x>0.
\]
\end{remark}

\begin{figure}
\begin{center}
\includegraphics[width=0.7\linewidth]{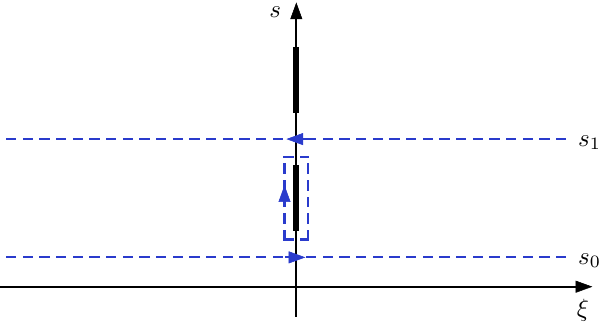}
\caption{By Cauchy's theorem and \eqref{eq:deformation estimate}, the integral of $e^{ix\zeta} m(\zeta)$ along the dashed contour vanishes. The filled intervals on the positive $s$-axis are the branch cuts $[\pi/2, \pi], [3\pi/2, 2\pi], \ldots$. Equation \eqref{eq:second deformation} is obtained by letting the inner contour converge to the branch cut which it surrounds.}
\label{fig:contour}
\end{center}
\end{figure}

\begin{remark}
We also remark that an alternative approach to obtaining the above positivity and monotonicity properties of the Whitham kernel is to study the functions $-x \Diff_x K(x)$ and $x^2 \Diff_x^2 K (x)$. These functions are regular at the origin and one can show that their Fourier transforms $\Diff_\xi(\xi m(\xi))$ and $\Diff_\xi^2 (\xi^2 m(\xi))$, respectively, are positive definite.
\end{remark}

We now improve upon Proposition~\ref{prop:decay} by taking advantage of  the expression \eqref{eq:explicit expression for Whitham kernel} for the kernel \(K\). For the technique behind our approach, we refer the reader to~\cite{CKP66}.

\begin{corollary}\label{cor:improved asymptotics}
The Whitham kernel satisfies
\[
K(x)=\frac{\sqrt{2}}{\pi\sqrt{|x|}} \exp(-{\textstyle \frac{\pi}{2}|x|}) +\bigOh(|x|^{-3/2}\exp(-{\textstyle \frac{\pi}{2}|x|}))
\]
as \(|x|\to \infty\).
\end{corollary}

\begin{proof}
Note that
\begin{equation}
\label{eq:K splitting}
K(x)=\frac1{\pi}\int_{\frac{\pi}{2}}^{\frac{3\pi}{4}} \exp(-s|x|) \sqrt{\frac{|\tan s|}{s}}\, \diff s 
+\int_{\frac{3\pi}{4}}^\infty \exp(-s|x|) g(s)\, \diff s,
\end{equation}
where
\[
g(s)=\frac1\pi\sum_{n=1}^\infty \sqrt{\frac{|\tan s|}{s}}\chi_{((2n-1)\pi/2, n\pi)}(s).
\]
The integral of \(g\) over each interval \(((2n-1)\pi/2, n\pi)\) can be estimated by the same constant (due to the periodicity of \(\tan s\)) and we therefore find that the second term in \eqref{eq:K splitting} is \(\bigOh(\exp(-{\textstyle \frac{3\pi}{4}}|x|))\)
as \(|x|\to \infty\). On the other hand, letting
\[
h(s)=\frac1{\pi}\sqrt{\frac{|\tan s|(s-\frac{\pi}{2})}{s}},
\]
which is smooth on the interval \( [\frac{\pi}{2}, \frac{3\pi}{4}]\), 
we can write the first term in \eqref{eq:K splitting} as 
\[
\int_{\frac{\pi}{2}}^{\frac{3\pi}{4}} \frac{\exp(-s|x|)}{\sqrt{s-\frac{\pi}{2}}} h(s)\, \diff s=
\frac{\exp(-{\textstyle \frac{\pi}{2}}|x|)}{\sqrt{|x|}}\int_0^{\frac{\pi}{4}|x|} \frac{\exp(-u)}{\sqrt{u}} h({\textstyle \frac{\pi}{2}+\frac{u}{|x|})}\, \diff u,
\]
where \(u=(s-\frac{\pi}{2})|x|\).
By the mean value theorem, we have that
\[
h({\textstyle \frac{\pi}{2}+\frac{u}{|x|})}=h({\textstyle \frac{\pi}{2}})+ \bigOh({\textstyle \frac{u}{|x|}}),
\]
uniformly for \(0\le u\le \frac{\pi}{4}|x|\). Estimating
\[
\int_0^{\frac{\pi}{4}|x|} \frac{\exp(-u)}{\sqrt{u}} \frac{u}{|x|}\, \diff u \le 
\frac1{|x|} \int_0^{\infty} \sqrt{u} \exp(-u)\, \diff u
\]
and
\[
\int_{\frac{\pi}{4}|x|}^\infty \frac{\exp(-u)}{\sqrt{u}}\, \diff u=\bigOh(\exp(-{\textstyle \frac{\pi}{4}} |x|)),
\]
we therefore obtain that 
\begin{align*}
\int_0^{\frac{\pi}{4}|x|} \frac{\exp(-u)}{\sqrt{u}} h({\textstyle \frac{\pi}{2}+\frac{u}{|x|})}\, \diff u
&=\int_0^{\infty} \frac{\exp(-u)}{\sqrt{u}} h({\textstyle \frac{\pi}{2}})\, \diff u
+\bigOh( |x|^{-1})\\
&=\frac{\sqrt{2}}{\pi}+\bigOh(|x|^{-1}),
\end{align*}
which concludes the proof.
\end{proof}

\section{The periodised Whitham kernel and the operator \(L\)}\label{sec:Kp}
We introduce the periodised Whitham kernel 
\begin{equation}\label{def:Kp}
K_{P}(x) = \sum_{n \in \Z} K(x + nP), 
\end{equation}
for \(P \in (0,\infty)\). 
Note that this sum is absolutely convergent, in view of that \(K\) has rapid decay. Note also that the evenness of $K$ is inherited by $K_P$.

Equivalently, $K_P$ can be expressed as the Fourier series
\[
K_P(x)=\frac{1}{P} \sum_{n\in \Z} m\left(\frac{2\pi n}{P}\right)\exp\left(\frac{2\pi i n x}{P}\right).
\]
For convenience we shall accept also \(P=\infty\), with the convention \(K_\infty = K\). The periodisation \(K_P\) is introduced to facilitate the analysis of periodic solutions satisfying certain sign conditions in a half-period.

Using the exponential decay of $K$ and all of its derivatives, one obtains directly the corresponding description of \(K_P\) (note here, though, that the singularity is repeated periodically at all integer multiples of \(P\)).

\begin{proposition}
\label{prop:periodic kernel decomposition}
The periodic Whitham kernel satisfies
\[
K_P(x)=\frac{1}{\sqrt{2\pi|x|}}+K_{P,\text{reg}}(x)
\]
where $K_{P, \text{reg}}$ is real analytic in $(-P, P)$. 
\end{proposition}
\begin{proof}
This follows by combining Proposition~\ref{prop:kernel decomposition} with the definition of \(K_P\), noting that one may differentiate termwise in \eqref{def:Kp} to arbitrary high order in view of Proposition~\ref{prop:decay}. 
\end{proof}

We obtain monotonicity results for \(K_P\) by applying the following general result concerning periodic kernels. The latter follows from Bernstein's theorem by noting that \(\diff \mu(s)/s\) is a finite measure (cf.~the proof of Proposition~\ref{prop:Stieltjes Fourier}), and by differentiating under the integral sign in the below formula for \(g_P\).

\begin{proposition}\label{prop:periodic formula}
Let \(g\in L^1(\R)\) be even and completely monotone. Then the periodisation 
\[
g_P(x) =\sum_{n\in \Z} g(x+nP)
\]
converges for each \(x \in \R \setminus P\Z\), and is given by
\[
g_P(x)= 
\int_{(0,\infty)} \frac{\cosh(s(x-\frac{P}{2}-P \lfloor \frac{x}{P}\rfloor))}{\sinh(\frac{P}{2} s)} \, \diff\mu(s),
\]
for \(x\in \R \setminus P\Z\), and \(g\)  the Laplace transform of \(\mu\).
Hence, \(g_P\) is smooth in \(\R \setminus P\Z\) and completely monotone on \((0, P/2)\). 
\end{proposition}

Combining Proposition~\ref{prop:periodic formula} with the formula \eqref{eq:explicit expression for Whitham kernel} for the Whitham kernel, we get the following result for its periodisation.

\begin{corollary}\label{cor:periodic formula}
The $P$-periodic Whitham kernel is given by
\[
K_P(x)=\frac{1}{\pi}
 \sum_{n=1}^\infty\int_{\frac{(2n-1)\pi}{2}}^{n\pi}\frac{\cosh(s(x-\frac{P}{2}-P \lfloor \frac{x}{P}\rfloor))}{\sinh(\frac{P}{2} s)} \sqrt{\frac{|\tan s|}{s}}\,\diff s,
\]
for \(x \in \R \setminus P\Z\).
Hence, \(K_P\) is smooth in \(\R \setminus P\Z\) and completely monotone on \((0, P/2)\). 
In particular, it is positive, strictly decreasing on \((0,P/2)\), and convex on \((0, P)\).
\end{corollary}

\begin{remark}
The monotonicity of $K_P$ in $(0,P/2)$ can in fact be proved using just the convexity and monotonicity of \(K\) (and the rapid decay of $K$ and its derivatives). Indeed, one has that
\begin{equation}\label{eq:monotone_periodic}
\begin{aligned}
\Diff_x K_P(x) &= \sum_{n \in \N} K'(x + nP)\\ 
&= \sum_{k \geq 0} \left( K'(x+kP) + K'(x-(k+1)P)\right).
\end{aligned}
\end{equation}
Let \(a_k = x + kP\) and \(b_k = x - (k+1)P\). Then \(K'(a_k) < 0\), whereas \(K'(b_k) > 0\), for all \(x \in (0,P/2)\) and all integers \(k \geq 0\). We thus want 
\[
|K'(a_k)| >  |K'(b_k)|. 
\]
By the evenness of \(K\), we have \(|K'(\zeta)| = |K'(-\zeta)|\) for any \(\zeta \neq 0\). And by  Proposition~\ref{prop:Whitham formula},
\(|\zeta| \mapsto |K'(|\zeta|)|\) is furthermore a strictly decreasing function of \(|\zeta|\), so that 
\[
|a_k| < (k+1/2)P < |b_k|
\] 
guarantees that \(|K'(a_k)| >  |K'(b_k)|\). Hence, the sum in \eqref{eq:monotone_periodic} is strictly negative for all \(x \in (0, P/2)\). Similarly, one may prove strict signs of higher-order derivatives of \(K_P\) on \((0,P/2)\) by using the signs of higher-order derivatives of \(K\).
\end{remark}

\subsection*{The operator \(L\)}
Now let \(L\) be the operator 
\[
L \colon f \mapsto K \ast f, 
\]
defined via duality on the space \(\Schwartz^\prime(\R)\) of tempered distributions. From the definition \eqref{def:Kp} of \(K_P\), one readily sees that for a  continuous periodic function \(f\), the operator \(L\)  is given by  \(\int_{-P/2}^{P/2} K_P(x-y)f(y)\,\diff y\), and more generally by \(\int_{\R} K(x-y)f(y)\,\diff y\) if \(f\) is bounded and continuous.  

Let \(H^s(\R)\), \(s \in \R\), denote the Sobolev (Bessel-potential) spaces with norm 
\[
\|f\|_{H^s(\R)} = \left(\int_\R   (1+k^2)^s  |\hat f(k)|^2 \, \diff k\right)^{1/2},
\]
and let \(H^s(\s_P)\) be the corresponding Sobolev spaces of \(P\)-periodic tempered distributions \(f = (1/P)\sum_{k \in \Z} \hat f_k \exp(i2\pi k \cdot /P)\) satisfying
\[
\|f\|^2_{H^s(\s_P)} = \sum_{k \in \Z} \left(1+\frac{4 \pi^2 k^2}{P^2}\right)^s |\hat f_k|^2 < \infty,
\]
where \(\s_P\) denotes the circle of circumference \(P>0\).
Note that \(H^0(\s_P)\) can be identified with \({L^2}(-P/2,P/2)\). For a nonnegative integer \(k\) we let \({BUC}^k(\R)\) be the space of $k$ times continuously differentiable functions on \(\R\), whose derivatives of order less than or equal to \(k\) are bounded and uniformly continuous on \(\R\).

We shall say that a function \(\varphi \colon \R \to \R\) is  \emph{H\"older continuous of regularity \(\alpha \in (0,1)\) at a point \(x \in \R\)}  if
\[
|\varphi|_{C^\alpha_x} := \sup_{h \ne 0} \frac{|\varphi(x+h) - \varphi(x)|}{|h|^\alpha} <  \infty,
\]
and let 
\[
C^\alpha(\R)=\{ \varphi \in BUC(\R) \colon \sup_{x}|\varphi|_{C_x^\alpha}<\infty\},
\]
\[
C^{k,\alpha}(\R)=\{\varphi \in BUC^k(\R) \colon \varphi^{(k)} \in C^\alpha(\R)\}.
\]
With \(C^{k, \alpha}(\s_P)\) we denote the closed subspace of \(C^{k, \alpha}(\R)\) consisting of functions that are \(P\)-periodic.

We also recall the definition of Besov spaces \(B_{p,q}^s(\R)\) using the Littlewood--Paley decomposition. 
Let \(\varrho \in C_0^\infty(\R)\) with \(\varrho(\xi)=1\) if \(|\xi|\le 1\), \(\varrho(\xi)=0\) if \(|\xi|\ge 2\), and define
\[
\gamma(\xi)=\varrho(\xi)-\varrho(2\xi),
\]
so that $\gamma\in C_0^\infty(\R)$ is supported in the set $1/2\le |\xi|\le 2$. We let
\[
\gamma_0(\xi)=\varrho(\xi)
\]
and
\[
\gamma_j(\xi)=\gamma(\xi/2^j), \quad j\ge 1,
\]
so that $\gamma_j$ is supported in the set $2^{j-1}\le |\xi|\le 2^{j+1}$ when $j\ge 1$ and $|\xi|\le 2$ when $j=0$, and
\[
\sum_{j=0}^\infty \gamma_j(\xi)=1, \quad \xi\in \R.
\]
For a tempered distribution \(f\in \SS'(\R)\) we let \(\gamma_j(D) f=\F^{-1}(\gamma_j(\xi) \hat f(\xi))\), so that
\[
f=\sum_{j=0}^\infty \gamma_j(D) f.
\]
The Besov spaces \(B_{p,q}^s(\R)\), \(s\in \R\), \(1\le p\le\infty\), \(1\le q<\infty\) are defined by
\[
\Bigg\{f \in \SS'(\R)\colon \|f\|_{B_{p,q}^s(\R)} :=\Big[\sum_{j=0}^\infty (2^{sj}\|\gamma_j(D) f\|_{L^p(\R)})^q\Big]^{\frac1q}<\infty\Bigg\}.
\]
For $1\le p\le \infty$ and $q=\infty$, we instead define
\[
B_{p,\infty}^s(\R) = \Bigg\{f \in \SS'(\R)\colon \|f\|_{B_{p,\infty}^s(\R)}:=\sup_{j\ge 0} 2^{sj}\|\gamma_j(D) f\|_{L^p(\R)}<\infty\Bigg\}.
\]
For a \(P\)-periodic tempered distribution \(f =  (1/P)\sum_{k \in \Z} \hat f_k \exp(i2\pi k \cdot /P)\), 
we have the identity
\[
\gamma_j(D) f=\frac1P \sum_{k \in \Z} \gamma_j\left(\frac{2\pi k}{P}\right) \hat f_k \exp\left(\frac{2\pi ik x}{P}\right),
\]
so that $\gamma_j(D) f$ is a trigonometric polynomial. The space $B_{p,q}^s(\s_P)$, $s\in \R$, $1\le p,q\le \infty$, is defined by replacing $\R$ by \(\s_P\) in the definition of $B_{p,q}^s(\R)$.
Note that \(B_{2,2}^s\) can be identified with \(H^s\), on the line as well as on the circle.

We furthermore define the Zygmund spaces \( {\mathcal C}^s\), \(s \in \R\), by
\[
{\mathcal C}^s(X)=B_{\infty, \infty}^s(X), \qquad X \in \{\R, \s_P\},
\]
and recall that \({\mathcal C}^s=C^{\lfloor s\rfloor, s-\lfloor s\rfloor}\)  for \(s \in \R_{> 0} \setminus \N\),
while \(W^{s, \infty} \subsetneq {\mathcal C}^s\) when $s$ is a nonnegative integer; both relations valid on the line as well as on the circle.
It follows from the estimate \(|D^n_\xi m(\xi)|\lesssim (1+|\xi|)^{-1/2-n}\), \(n\ge 0\), that \(L\) defines a bounded operator
\[
L\colon B_{p,q}^s(X)\to B_{p,q}^{s+\frac12}(X), \qquad  X \in \{\R,\s_P\},
\]
see, e.g.,~\cite{MR2064734, BahouriCheminDanchin11}.
In particular, the operators
\begin{align*}
L\colon H^s(X) \to H^{s+\frac{1}{2}}(X) \quad &\text{and} \quad L\colon {\mathcal C}^s(X)\to {\mathcal C}^{s+\frac{1}{2}}(X)
\end{align*}
are bounded  on \(\R\) as well as on \(\s_P\), for all \(s \in \R\). For an introduction to periodic distributions and function spaces, we refer the reader to Chapter 9 in the monograph \cite{MR781540} by Triebel.

\subsection*{Notational conventions}
To ease notation in what follows, when \(f(x) > g(x)\) for all \(x\) we write \(f > g\), and when \(f(x) \geq g(x)\) for all \(x\) with \(g(x_0) > f(x_0)\) for some \(x_0\) we write \(f \gneq g\). Similarly,  \(f \geq g\)  indicates that \(f(x) \geq g(x)\) for all \(x\), and \(f = g\)  that \(f\) and \(g\) are identically equal. We make the corresponding conventions for the relations \(<\), \(\lneq\), and \(\leq\). Finally, \(f \neq g\) denotes the situation when  \(f(x_0) \neq g(x_0)\) for some \(x_0\).

\begin{lemma}\label{lemma:Lmonotone}
\(L\) is strictly monotone:  \(Lf > Lg\) if \(f\) and \(g\) are bounded and continuous functions with \(f \gneq g\).
\end{lemma}

\begin{proof}
This is immediate from the strict positivity of \(K\) and \(K_P\), see Proposition~\ref{prop:Whitham formula} and Corollary~\ref{cor:periodic formula}.
\end{proof}

\begin{lemma}\label{lemma:parity}
The operator \(L\) is parity-preserving on any period \(P \in (0,\infty]\), and \(L f(x) > 0\) on \((-P/2,0)\) for \(f\) \(P\)-periodic, odd and continuous with \(f \gneq 0\) on \((-P/2,0)\).
\end{lemma}

\begin{proof}
To see that \(L\) is parity-preserving, note that
\begin{align*}
Lf(x) \pm Lf(-x) &= \int_{-P/2}^{P/2} K_P(x-y)f(y)\,\diff y \pm \int_{-P/2}^{P/2} K_P(-x-y)f(y)\,\diff y\\
&= \int_{-P/2}^{P/2} K_P(x-y) \left( f(y) \pm f(-y)\right)\,\diff y,
\end{align*}
which vanishes for \(f\) odd (even). 

Next, assume that \(f\) is \(P\)-periodic, odd and continuous, with \(f(x) \geq 0\) for \(-P/2 \leq x \leq 0\) and \(f(x) \neq 0\) for some \(x\). Then
\begin{equation}\label{eq:Lf}
\begin{aligned}
Lf(x) &= \int_{-P/2}^{P/2} K_P(x-y)f(y)\,\diff y\\
&= \int_{-P/2}^0 \left(K_P(x-y) - K_P(x+y)\right) f(y)\,\diff y.
\end{aligned}
\end{equation}
Fix \(x \in (-P/2,0)\), and consider first the case when \(P=\infty\). We have
\[
|x+y| = |x| + |y| >  |x - y|, \qquad y \in (-P/2,0),
\] 
so that the distance from the origin to the point \(x+y\) is larger than that to the point \(x-y\). Since \(K\) is even and strictly decreasing as a function of the distance to the origin, we find that \(K(x-y) > K(x+y)\), which proves the desired conclusion.

When \(P < \infty\), fix again \(x\ \in (-P/2,0)\) and consider \(y\) such that
\[
-P<x+y\leq x-y <P/2.
\]
This covers all possible values of \(x+y\) and \(x-y\) appearing in the last integral in \eqref{eq:Lf}. Since \(K_P\) decreases with the distance to the origin in the period \((-P/2,P/2)\), and is periodic with period \(P\), all that remains is to convince ourselves that
\[
 \dist(x-y,0) < \min\{\dist(x+y,0),\dist(x+y,-P)\}.
\]
The inequality \(|x-y| < |x+y|\) holds as above for all same-signed \(y \neq x\), as does \(|x-y| < P+x+y\) for all \(x,y > -P/2\). This proves that  \(K_P(x-y) > K_P(x+y)\) almost everywhere in the interval, and therefore \(Lf(x) > 0\) when \(-P/2 < x < 0\). (Note that when \(x\) is a multiple of \(P\) the same argument fails, because \(K_P\) is even around those points.)
\end{proof}


\section{Nodal pattern}\label{sec:nodal}

In this section we record some basic properties of Whitham solutions, including a priori bounds, regularity estimates and a maximum principle. This will enable us to establish a nodal pattern for solutions of the steady Whitham equation, recorded in Theorem~\ref{thm:nodal}. It is interesting to note how the equation \eqref{eq:steadywhitham} features many of the properties of elliptic equations.

We remind the reader that with a solution of the steady Whitham equation we mean a real-valued, continuous and bounded function \(\varphi\) that satisfies \eqref{eq:steadywhitham} pointwise. \emph{In the case \(P < \infty\), we presuppose that any solution \(\varphi\) is \(P\)-periodic.} We shall furthermore call a continuous and bounded function \(\varphi\) a \emph{supersolution} of \eqref{eq:steadywhitham} if 
\[
-\mu \varphi + L\varphi + \varphi^2 \leq 0.
\]
Similarly, we call \(\varphi\) a \emph{subsolution} of \eqref{eq:steadywhitham} if \(-\mu \varphi + L\varphi + \varphi^2 \geq 0\).

\begin{lemma}\label{lemma:apriori_1}
Let \(I_\mu\) be the closed interval with endpoints \(\mu -1\) and \(0\). Then supersolutions \(\varphi_1\) and subsolutions \(\varphi_2\) of the steady Whitham equation \eqref{eq:steadywhitham} satisfy
\[
\inf \varphi_1 \in I_\mu \quad\text{ and }\quad \sup \varphi_2 \not\in \mathrm{int}(I_\mu),
\]
where \(\mathrm{int}(I_\mu)\) is the interior of the interval \(I_\mu\). In particular, if \(\varphi\) is a solution, then either  \(\mu-1\le \inf \varphi \le 0\le \sup\varphi\) or \(\varphi(x)\equiv \mu-1\) if \(\mu\le 1\), while 
either \(0\le \inf \varphi \le \mu-1\le \sup\varphi\) or \(\varphi(x) \equiv 0\) if \(\mu>1\).
\end{lemma}

\begin{remark}\label{rem:sign-changing}
One can see directly from \eqref{eq:steadywhitham} that if a solution satisfies \(\varphi(x) = 0\) for some \(x\), then \(\varphi\) is either identically zero or it changes sign. Indeed, at \(x\) the equation reduces to \(L\varphi = 0\). Since \(L\) is a strictly monotone operator this is impossible unless \(\varphi\) is sign-changing or vanishes everywhere.
\end{remark}

\begin{proof}
For \(\varphi_1\) a supersolution, we have \(\left( \varphi_1 - \frac{\mu}{2} \right)^2 \leq \frac{\mu^2}{4} - L\varphi_1\). By Lemma~\ref{lemma:Lmonotone}, \(L\) is a strictly monotone operator. Since furthermore \(L c = c\) for constants \(c\), we therefore obtain that
\[
\left( \varphi_1 - \frac{\mu}{2} \right)^2 \leq \frac{\mu^2}{4} - \inf \varphi_1.
\]
In particular, \(\left( \inf \varphi_1 - \frac{\mu}{2} \right)^2 \leq \frac{\mu^2}{4} - \inf \varphi_1\), and therefore
\[
(\inf \varphi_1) \left( \inf \varphi_1 - (\mu -1)\right) \leq 0.
\]
Similarly, for \(\varphi_2\) a subsolution one obtains
\(
\left( \varphi_2 - \frac{\mu}{2} \right)^2 \geq \frac{\mu^2}{4} - L\varphi_2 \geq \frac{\mu^2}{4} - \sup\varphi_2,
\)
which yields that \((\sup \varphi_2) \left( \sup \varphi_2 - (\mu -1)\right) \geq 0\).
\end{proof}

Since a solution is simultaneously a subsolution and a supersolution, in that case we obtain from Lemma \ref{lemma:apriori_1} that either 
\(\inf \varphi=\sup\varphi=\mu-1\) or \(\sup \varphi \ge 0\) when \(\mu \le 1\).  When \(\mu>1\), we obtain instead that either \(\inf \varphi=\sup\varphi=0\) or \(\sup \varphi\ge \mu-1\).

The following lemma is the equivalent of  the strong maximum principle for elliptic equations.

\begin{lemma}[Touching lemma]\label{lemma:touching}
Let \(\varphi_1\) be a supersolution and \(\varphi_2\) a subsolution of the steady Whitham equation \eqref{eq:steadywhitham} with \(\varphi_1 \geq \varphi_2\). Then either 
\begin{itemize}
\item[(i)] \(\varphi_1 = \varphi_2\), or
\item[(ii)] \(\varphi_1 > \varphi_2\) with \(\varphi_1 + \varphi_2 < \mu\).
\end{itemize}
\end{lemma}

\begin{proof}
Since \(\varphi_1, \varphi_2\) are super- and subsolutions, respectively, we have that,
\[
(\mu-L) (\varphi_1 - \varphi_2) \geq  (\varphi_1 + \varphi_2) (\varphi_1 - \varphi_2).
\]
If \(\varphi_1 = \varphi_2\) there is nothing to prove, so assume that \(\varphi_1\gneq \varphi_2\). Since \(L\) is a strictly monotone operator, we then see that
\[
\left( \mu - (\varphi_1 + \varphi_2) \right) (\varphi_1 - \varphi_2) \geq L(\varphi_1 - \varphi_2) > 0.
\]
Thus \(\varphi_1(x) \neq \varphi_2(x)\) and \(\mu \neq \varphi_1(x) + \varphi_2(x)\) for all \(x\). In view of that \(\varphi_1 \geq \varphi_2\) by assumption, so that  \(\varphi_1 - \varphi_2\) and \(L(\varphi_1 - \varphi_2)\) therefore have the same sign,  we conclude that \(\varphi_1 > \varphi_2\) and \(\varphi_1 + \varphi_2 < \mu\).
\end{proof}

\begin{corollary}\label{cor:mubound}
Except for the trivial solutions \(\varphi = 0\) and \(\varphi = \mu -1\), supersolutions \(\varphi\) of the steady Whitham equation \eqref{eq:steadywhitham} satisfy 
\begin{align*}
\mu -1 < \varphi < 1, &\qquad  \mu < 1,\\
0 < \varphi < \mu &\qquad \mu > 1.
\end{align*}
\end{corollary}

\begin{remark}
For \(\mu = 1\), the solution \(\varphi = 0\) is the unique integrable supersolution, cf.~Proposition~\ref{prop:mean} below.
\end{remark}

\begin{proof}
For any \(\mu \in \R\), the functions \(x \mapsto \mu -1 \) and \(x \mapsto 0\) are constant solutions of the Whitham equation. 

For \(\mu < 1\), Lemma~\ref{lemma:apriori_1} guarantees that \(\varphi \geq \mu - 1\) for any supersolution \(\varphi\). Thus, we may apply Lemma~\ref{lemma:touching} with \(\varphi_1 = \varphi\) and \(\varphi_2 = \mu -1\) to conclude that \(\varphi > \mu -1 \) and, furthermore, that \(\varphi + \mu -1 < \mu\), meaning that \(\varphi < 1\).

Similary, for \(\mu > 1\) we have \(\varphi \geq 0\) for any supersolution \(\varphi\) by Lemma~\ref{lemma:apriori_1}, and the rest of the conclusion follows from Lemma~\ref{lemma:touching}.
\end{proof}

\begin{proposition}\label{prop:mean}
Any solution \(\varphi \in {L^1}(-P/2,P/2)\) of the steady Whitham equation \eqref{eq:steadywhitham} belongs to \({L^2}(-P/2,P/2)\) and has negative or positive mean according to whether \(\mu < 1\) or \(\mu > 1\). More precisely,
\[
(\mu -1)\int_{-P/2}^{P/2} \varphi(x) \,\diff x = \|\varphi\|_{{L^2}(-P/2,P/2)}^2 
\]
holds for any \(\mu \in \R\), where \(P \in (0,\infty]\) is the possibly infinite period.
\end{proposition}

\begin{proof}
By integrating \((\mu - L)\varphi = \varphi^2\) over a period, we get
\[
\mu \int_{-P/2}^{P/2} \varphi(x)\,\diff x - \int_{-P/2}^{P/2} L\varphi(x)\,\diff x = \int_{-P/2}^{P/2} \left(\varphi(x)\right)^2\,\diff x.
\]
Consider first \(P < \infty\). In view of that \(\int_{-P/2}^{P/2} \varphi(x)\,\diff x = \hat \varphi_0\), we have that \((\mu - m(0)) \hat\varphi_0 = \|\varphi\|_{{L^2}(-P/2,P/2)}^2\). For \(P = \infty\), replace \(\hat \varphi_0\)  by \(\hat \varphi(0)\).
\end{proof}

The following result improves upon Lemma~\ref{lemma:touching} in the case when one has additional control of the first-order derivatives of the solutions. An important consequence of it is Theorem~\ref{thm:nodal}.

\begin{lemma}[Touching lemma for derivatives]\label{lemma:touching2}
Let \(\varphi_1, \varphi_2\) be even and continuously differentiable solutions of the steady Whitham equation \eqref{eq:steadywhitham} with \(\varphi_1 \geq \varphi_2\) and \(\varphi_1^\prime \gneq \varphi_2^\prime \geq 0\) in \((-P/2,0)\). Then \(\varphi_1^\prime > \varphi_2^\prime\) and \(\varphi_1 + \varphi_2 < \mu\) in \((-P/2,0)\).
\end{lemma}

\begin{remark}
It follows from Lemma~\ref{lemma:touching2} that \(\varphi_1 > \varphi_2\) on the whole interval \((-P/2,P/2)\). This is due to the evenness of \(\varphi_1, \varphi_2\) and the strict inequality \(\varphi_1^\prime > \varphi_2^\prime\) on \((-P/2,0)\).
\end{remark}

\begin{proof}
Both \(\varphi_1\) and \(\varphi_2\) solve \((\mu - L)\varphi = \varphi^2\), so we may subtract and differentiate to obtain
\begin{equation}\label{eq:mu-L}
(\mu - L) (\varphi_1^\prime - \varphi_2^\prime) = (\varphi_1^2 - \varphi_2^2)^\prime.
\end{equation}
Since \( (\varphi_1^\prime + \varphi_2^\prime)(\varphi_1 - \varphi_2)  \geq 0\) on \((-P/2,0)\) by assumption, we see by expanding the right-hand side of \eqref{eq:mu-L} that
\[
(\mu -\varphi_1  - \varphi_2) (\varphi_1^\prime - \varphi_2^\prime) \geq L (\varphi_1^\prime - \varphi_2^\prime) > 0 \qquad \text{ on }\quad  (-P/2,0),
\]
where Lemma~\ref{lemma:parity} has been used with \(f=\varphi_1'-\varphi_2'\). Because \(\varphi_1^\prime \geq \varphi_2^\prime\) on \((-P/2,0)\), this implies both that \(\varphi_1^\prime > \varphi_2^\prime\) and that \(\varphi_1 + \varphi_2 < \mu\) on that interval. 
\end{proof}

We have now come to the main result of this section, which we shall later need to prove that the global bifurcation branch of steady solutions does not form a closed loop.

\begin{theorem}[Nodal pattern]\label{thm:nodal}
Let \(P \in (0,\infty]\). Any \(P\)-periodic, nonconstant and even solution \(\varphi \in BUC^1(\R)\) of the steady Whitham equation \eqref{eq:steadywhitham} which is nondecreasing on \((-P/2,0)\) satisfies 
\[\varphi' > 0, \: \varphi < \frac{\mu}{2} \qquad\text{ on }\quad (-P/2,0).
\]
For such a solution one necessarily has \(\mu>0\).

If furthermore \(\varphi \in BUC^2(\R)\), then \(\varphi < \frac{\mu}{2}\) everywhere and
\[ 
\varphi^{\prime\prime}(0) < 0.
\]
For \(P < \infty\) one has  \(\varphi^{\prime\prime}(\pm P/2) > 0\). If in addition  \(\mu \leq 1\) and \(\varphi(0) \geq  \frac{\mu}{4}\), then 
\[
\varphi^{\prime\prime}(\textstyle{\frac{P}{2}}) -  \varphi^{\prime\prime}(0)\\ \geq   \frac{1}{2} |K_P^\prime(\frac{P}{4})|. 
\]
\end{theorem}

\begin{proof}
To prove that  \(\varphi' > 0\) and \(\varphi < \frac{\mu}{2}\) on \((-P/2,0)\), note first that by assumption \(\varphi'\) must be odd, nontrivial, and nonnegative in \((-P/2,0)\). According to Lemma~\ref{lemma:parity}, we then have \(L\varphi' > 0\) in \((-P/2,0)\), and by \((\mu - 2 \varphi) \varphi' = L \varphi'\) also
\[
\varphi' (\mu - 2\varphi) > 0	 \qquad\text{ in } \quad (-P/2,0).
\] 
The sign of \(\mu - 2 \varphi\) can then be inferred from that of \(\varphi'\). This proves that \(\varphi^\prime > 0\) and \(\varphi < \frac{\mu}{2}\) on the open half-period \((-P/2,0)\). On the other hand, since \(\varphi\) is nonconstant it follows from Lemma~\ref{lemma:apriori_1} and Remark~\ref{rem:sign-changing} that \(\varphi(0)=\sup \varphi>0\), so that \(\mu>0\).

Now suppose that \(\varphi \in BUC^2(\R)\). To show that \(\varphi^{\prime\prime}(0)\) is strictly negative, we differentiate the equation twice to obtain that
\[
(\mu - 2\varphi) \varphi^{\prime\prime} = 2(\varphi^\prime)^2 + L \varphi^{\prime\prime}. 
\]
Evaluating this equality at \(x = 0\) using the evenness of \(K\) and \(\varphi\), we see that
\begin{align*}
\left(\frac{\mu}{2} - \varphi(0)\right) \varphi^{\prime\prime}(0) &= \int_{0}^{P/2} K_P(y) \varphi^{\prime\prime}(y)\,\diff y\\
& = \int_{0}^{{\varepsilon}} K_P(y) \varphi^{\prime\prime}(y)\,\diff y + \int_{{\varepsilon}}^{P/2} K_P(y) \varphi^{\prime\prime}(y)\,\diff y\\
&= \int_{0}^{{\varepsilon}} K_P(y) \varphi^{\prime\prime}(y)\,\diff y + \left[K_P(y) \varphi^{\prime}(y)\right]_{y= {\varepsilon}}^{y=P/2}\\ 
&\quad -\int_{{\varepsilon}}^{P/2} K_P^\prime(y) \varphi^{\prime}(y)\,\diff y.
\end{align*}
Because \(\varphi^{\prime\prime}\) is continuous and  \(K_P\) is integrable, with \(K_P(x) \sim |x|^{-\frac{1}{2}}\) for \(|x| \ll 1\), the first integral vanishes as \({\varepsilon} \to 0\). The boundary term \(K_P(P/2) \varphi^\prime(P/2)\) vanishes in view of that \(\varphi^\prime(P/2)=0\) if \(P<\infty\) and because \(\lim_{x\to \infty} K(x)=0\) and \(\varphi'\) is bounded if \(P=\infty\). Due to the regularity and evenness of \(\varphi\), we also have \(\varphi^\prime({\varepsilon}) = O({\varepsilon})\) so that \(\varphi^\prime({\varepsilon}) K_P({\varepsilon}) = O({\varepsilon}^{\frac{1}{2}}) \to 0\) as \({\varepsilon} \to 0\). 

By Corollary~\ref{cor:periodic formula} and what we just proved, both \(K_P^\prime\) and \(\varphi^\prime\) are strictly negative on \((0,P/2)\). Thus \(-\int_{{\varepsilon}}^{P/2} K_P^\prime(y) \varphi^{\prime}(y)\,\diff y\) is negative for any \({\varepsilon} > 0\), and strictly decreasing as \({\varepsilon} \searrow 0\) (note also that \(K_P\) is smooth in a vicinity of \(P/2\), see Proposition~\ref{prop:periodic kernel decomposition}). Thus, we may let \({\varepsilon} \searrow 0\) to see that
\begin{align*}
\left(\frac{\mu}{2} - \varphi(0)\right) \varphi^{\prime\prime}(0) &= - \lim_{{\varepsilon} \searrow 0} \int_{{\varepsilon}}^{P/2} K_P^\prime(y) \varphi^{\prime}(y)\,\diff y  < 0.
\end{align*}
Since \(\varphi\) is continuous with \(\varphi < \frac{\mu}{2}\) on \((-P/2,0)\), this proves that \(\varphi < \frac{\mu}{2}\) everywhere, and that \(\varphi^{\prime\prime}(0) < 0\).

When \(P < \infty\), note that \(K_P(P/2 -y) = K_P(-P/2 - y) = K_P(y + P/2)\), so that
\allowdisplaybreaks
\begin{align*}
\left(\frac{\mu}{2} -  \varphi(P/2)\right) \varphi^{\prime\prime}(P/2) &= \int_{0}^{P/2} K_P(y + P/2) \varphi^{\prime\prime}(y)\,\diff y\\
& =  \left( \int_{0}^{P/2- {\varepsilon}} + \int_{P/2- {\varepsilon}}^{P/2} \right) K_P(y + P/2) \varphi^{\prime\prime}(y)\,\diff y \\ 
&=   \left[K_P(y+P/2)\varphi^\prime(y) \right]_{y=0}^{y=P/2- {\varepsilon}}\\ 
&\quad+  \int_{P/2- {\varepsilon}}^{P/2} K_P(y + P/2) \varphi^{\prime\prime}(y)\,\diff y \\ 
&\quad-  \int_0^{P/2- {\varepsilon}} K_P^\prime(y + P/2) \varphi^{\prime}(y)\,\diff y. 
\end{align*}
By the same arguments as above all terms but the last on the right-hand side vanish as \({\varepsilon} \searrow 0\), and the term \(-\int_0^{P/2- {\varepsilon}} K_P^\prime(y + P/2) \varphi^{\prime}(y)\,\diff y\) is strictly positive and increasing as \({\varepsilon} \searrow 0\). Thus \(\varphi^{\prime\prime}(P/2) > 0\).

To prove the final estimate, note that 
\begin{equation}\label{eq:firstbound}
\begin{aligned}
&( \textstyle{\frac{\mu}{2}} - \varphi(P/2)) \varphi^{\prime\prime}(P/2) - ( \textstyle{\frac{\mu}{2}} - \varphi(0)) \varphi^{\prime\prime}(0)\\ &= \int_{-P/2}^{0} (K_P^\prime(y) - K_P^\prime(y + P/2)) \varphi^\prime(y)\,\diff y\\
& \geq \min_{x \in [-P/4,0]} K_P^\prime(x) \int_{-P/2}^{0} \varphi^\prime(y)\,\diff y\\
& =  |K_P^\prime(P/4)| (\varphi(0)- \varphi(P/2)), 
\end{aligned}
\end{equation}
since \(K_P\) is even and strictly convex on \((-P/2,P/2)\). We rewrite \eqref{eq:firstbound} as
\[
( \textstyle{\frac{\mu}{2}} - \varphi(0)) (\varphi^{\prime\prime}(\frac{P}{2}) -  \varphi^{\prime\prime}(0))\\ \geq  \left(|K_P^\prime(\frac{P}{4})| - \varphi^{\prime\prime}(\frac{P}{2})\right) (\varphi(0)- \varphi(\frac{P}{2})).
\]
Now, either 
\(
\varphi^{\prime\prime}(\frac{P}{2}) \geq \frac{1}{2} |K_P^\prime(\frac{P}{4})|,
\)
or 
\[
|K_P^\prime(\textstyle{\frac{P}{4}})| - \varphi^{\prime\prime}(\textstyle{\frac{P}{2}}) \geq  \frac{1}{2} |K_P^\prime(\frac{P}{4})|.
\] 
In the second case, note first that one has  \( \varphi(0)- \varphi(P/2) \geq \frac{\mu}{4} \geq  \textstyle{\frac{\mu}{2}} - \varphi(0)\)  by the assumptions and the fact that nonconstant solutions with \(\mu \leq 1\) are sign-changing (cf.~Lemma~\ref{lemma:apriori_1} and Remark~\ref{rem:sign-changing}). Using this estimate, and dividing by \(\frac{\mu}{2} - \varphi(0) \le \varphi(0)- \varphi(P/2)\), we see that in either case
\[
\varphi^{\prime\prime}(\textstyle{\frac{P}{2}}) -  \varphi^{\prime\prime}(0)\\ \geq  \frac{1}{2} |K_P^\prime(\frac{P}{4})|.
\]
\end{proof}

\section{About the singularity at \(\varphi = \mu/2\)}\label{sec:singularity}

We now move on to investigate the case when a solution touches the value \(\frac{\mu}{2}\) from below. 
We begin by noting that a solution is smooth as long as it remains bounded away from \(\frac{\mu}{2}\) (recall that by a solution we mean a continuous and bounded solution).

\begin{theorem}[Regularity I]\label{thm:regularity I}
Let \(\varphi \leq \frac{\mu}{2}\) be a solution of the steady Whitham equation \eqref{eq:steadywhitham}.  Then:\\[-6pt]
\begin{itemize}
\item[(i)] If \(\varphi <\mu/2\) uniformly on \(\R\), then \(\varphi \in C^\infty(\R)\) and all of its derivatives are uniformly bounded on $\R$. 
\item[(ii)] If \(\varphi<\mu/2\) uniformly on \(\R\) and \(\varphi \in L^2(\R)\), then \(\varphi \in H^\infty(\R)\).
\item[(iii)] \(\varphi\) is smooth on any open set where \(\varphi<\mu/2\).
\end{itemize}
\end{theorem}

\begin{proof}
Assume first that \(\varphi < \frac{\mu}{2}\), uniformly on \(\R\). The operator \(L\) maps $B_{p,q}^s(\R)$ into 
$B_{p,q}^{s+1/2}(\R)$, $L^\infty(\R)\subset B_{\infty, \infty}^0(\R)$ into ${\mathcal C}^{1/2}(\R)=B_{\infty, \infty}^{1/2}(\R)\subset L^\infty(\R)$, and the Nemytskii operator 
\[
u \mapsto \mu/2 - \sqrt{\mu^2/4 - u}
\] 
maps \(B_{p,q}^s(\R)\cap L^\infty(\R)\) into itself for \(u < \frac{\mu^2}{4}\) and \(s > 0\)  (see \cite[Theorem 2.87] {BahouriCheminDanchin11}). All three mappings are continuous.
Since \(\varphi < \frac{\mu}{2}\), it follows that \(L\varphi < \frac{\mu^2}{4}\), and therefore
\begin{equation}\label{eq:bootstrapping}
 [L\varphi \mapsto \mu/2 - \sqrt{\textstyle{\frac{\mu^2}{4}} - L\varphi}] \circ [\varphi \mapsto L\varphi] \colon  B^s_{p,q}(\R)\cap L^\infty(\R) \hookrightarrow B^{s+\frac{1}{2}}_{p,q}(\R),
\end{equation}
for all \(s \ge 0\). Hence, the equality \(\varphi = \frac{\mu}{2} - \sqrt{\frac{\mu^2}{4} - L\varphi}\) guarantees that \(\varphi \in C^\infty(\R)\) with uniformly bounded derivatives as long as \(\varphi \in L^\infty(\R)\) (take $p=q=\infty$). This proves (i).
Taking $p=q=2$ proves (ii).

Now, if \(\varphi\) is in \(L^\infty(\R)\) and \({\mathcal C}_\text{loc}^s\) on an open set \(U\) in the sense that \(\psi \varphi\in {\mathcal C}^s(\R)\) for any \(\psi \in C^\infty_0(U)\), we still get that \(L\varphi\) is 
\({\mathcal C}_\text{loc}^{s+1/2}\)  in \(U\). Indeed, let $\psi\in C_0^\infty(U)$ and let $\tilde \psi\in C_0^\infty(U)$ be a smooth cut-off function with $\tilde\psi=1$ in a neighbourhood \(V \Subset U\) of $\supp \psi$. Then 
\[
\psi L\varphi=\psi L(\tilde \psi \varphi)+\psi L((1-\tilde \psi)\varphi). 
\]
The first term on the right-hand side is of class \(\mathcal{C}^{s+1/2}\). On the other hand, the second term is given by
\[
\int_{-\infty}^{\infty} K(x-y)\psi(x)(1-\tilde \psi(y))\varphi(y)\, \diff y.
\]
where the integrand vanishes for \(y\) near \(x\); it is therefore smooth. Hence, \(L\varphi\) is \(\mathcal{C}^{s+1/2}_\text{loc}\) in \(U\) and by the above iteration argument, if \(\varphi<\mu/2\) in \(U\) it is also smooth there. This proves (iii).
\end{proof}

The following lemma is essential in showing that solutions which touch \(\mu/2\) from below are not smooth.

\begin{lemma}\label{lemma:essentialbound}
Let \(P < \infty\), and let \(\varphi\) be an even, nonconstant solution of the steady Whitham equation \eqref{eq:steadywhitham} 
such that \(\varphi\) is nondecreasing on \((-P/2, 0)\) 
with \(\varphi \le \frac{\mu}{2}\). Then there exists a universal constant \(\lambda_{K,P} > 0\), depending only on the kernel \(K\) and the period \(P\), such that
\begin{equation}\label{eq:essentialbound}
\textstyle{\frac{\mu}{2}} - \varphi(\textstyle{\frac{P}{2}}) \geq \lambda_{K,P}.
\end{equation}
More generally, 
\begin{equation}\label{eq:essentialbound II}
\textstyle{\frac{\mu}{2}} - \varphi(x) \gtrsim_{K,P} |x_0|^{1/2}
\end{equation}
uniformly for all \(x \in [-P/2,x_0]\), with \(x_0 < 0\). 
\end{lemma}

\begin{remark}\label{remark:essentialbound}
By inspecting the proof of the lemma, one finds that estimate \eqref{eq:essentialbound II} is uniform in $P\gg 1$ and that it also also holds in the limiting case $P=\infty$ (for $x\in (-\infty, x_0]$).
\end{remark}

\begin{proof}
For the sake of clearness we prove \eqref{eq:essentialbound} first under the assumption that \(\varphi(0)<\mu/2\) (which implies that \(\varphi<\mu/2\) uniformly on \(\R\)). Subsequently, a short analogous argument is given for the general estimate  \eqref{eq:essentialbound II} under the same assumption. Finally, it is shown how to modify the argument to allow for \(\varphi(0)=\mu/2\). 

Note first that, by Theorem~\ref{thm:regularity I}~(i),  \(\varphi\) is smooth under the assumption \(\varphi(0)<\mu/2\).
Let \(x \in [-\frac{3P}{8}, -\frac{P}{8}]\). For a solution \(\varphi\) as in the assumptions, one has
\begin{equation}\label{eq:essentialapproach}
\begin{aligned}
(\textstyle{\frac{\mu}{2}} - \varphi(P/2)) \varphi^\prime(x) &\geq (\textstyle{\frac{\mu}{2}} - \varphi(x)) \varphi^\prime(x)\\
 &= \frac{1}{2} \int_{-P/2}^{P/2} K_P(x-y) \varphi^\prime(y) \, \diff y\\
 &= \frac{1}{2} \int_{-P/2}^{0} ( K_P(x-y) - K_P(x+y))  \varphi^\prime(y) \, \diff y\\
 &\geq \frac{1}{2} \int_{-3P/8}^{-P/8} ( K_P(x-y) - K_P(x+y))  \varphi^\prime(y) \, \diff y,
\end{aligned}
\end{equation}
in view of that \(K_P(x-y) > K_P(x+y)\) for  \(x,y \in (-\frac{P}{2},0)\). In fact, there exists a universal constant \(\tilde \lambda_{K,P} > 0\) depending only on the kernel \(K\) and the period \(P < \infty\), such that
\[
\min  \left\{ K_P(x-y) - K_P(x+y) \colon x,y \in [-\textstyle\frac{3P}{8},-\textstyle\frac{P}{8}] \right\} \geq \tilde \lambda_{K,P}.
\]
Thus, integrating \eqref{eq:essentialapproach} in \(x\) over the interval \((-\frac{3P}{8}, - \frac{P}{8})\) yields
\begin{align*}
&(\textstyle{\frac{\mu}{2}} - \varphi(\frac{P}{2})) (\varphi(-\frac{P}{8}) - \varphi(-\frac{3P}{8}))\\ 
&\geq \frac{1}{2} \int_{-3P/8}^{-P/8} \left( \int_{-3P/8}^{-P/8} ( K_P(x-y) - K_P(x+y)) \, \diff x\right)  \varphi^\prime(y) \, \diff y\\
 & \geq \textstyle{\frac{P}{8}} \tilde \lambda_{K,P}   (\varphi(-\frac{P}{8}) - \varphi(-\frac{3P}{8}) ). 
\end{align*}
Now, according to Theorem~\ref{thm:nodal}, \(\varphi(-\frac{3P}{8})  < \varphi(-\frac{P}{8})\) for a solution \(\varphi\) as in the assumptions, whence we may divide with \( \varphi(-\frac{P}{8}) - \varphi(-\frac{3P}{8})\) to conclude that
\[
\textstyle{\frac{\mu}{2}} - \varphi(\frac{P}{2}) \geq \textstyle{\frac{P}{8}} \tilde \lambda_{K,P} := \lambda_{K,P}.
\]

For the \(x\)-dependent estimate (\(\xi\) will here play the role of \(x\)), fix \(x_1, x_2\) with \(-P/4 < x_2 < x_1 < 0\), let \(x \in (x_2, x_1)\) and consider \(\xi \in [-P/2,x_2]\). Then
\begin{equation}
\begin{aligned}
(\textstyle{\frac{\mu}{2}} - \varphi(\xi)) \varphi^\prime(x) &\geq \frac{1}{2} \int_{x_2}^{x_1} ( K_P(x-y) - K_P(x+y))  \varphi^\prime(y) \, \diff y\\
&\geq \frac{1}{2} \int_{x_2}^{x_1} (-2y) K_P^\prime(y + \zeta) \varphi^\prime(y) \, \diff y\\
&\geq -x_1 K_P^\prime(2 x_2)  \left( \varphi(x_1) - \varphi(x_2) \right),
\end{aligned}
\end{equation}
where \(|\zeta| < |x|\) arises from the mean value theorem. Integrating over \((x_2,x_1)\) in \(x\), and dividing out then yields
\[
\textstyle{\frac{\mu}{2}} - \varphi(\xi) \geq |x_1(x_2 - x_1)| K_P^\prime(2 x_2).
\]
Now let \(x_2 = x_0\) and \(x_1 = x_0 /2\) to obtain that
\[
\textstyle{\frac{\mu}{2}} - \varphi(\xi) \geq \frac{1}{4} x_0^2  K_P^\prime(2 x_0) \gtrsim_{K,P} |x_0|^{1/2},
\]
because \(K_P^\prime(x) \sim |x|^{-3/2}\) for \(0 < -x \ll 1\), in view of Proposition~\ref{prop:periodic kernel decomposition}.

In the case when \(\varphi(0)=\frac{\mu}{2}\),  \(\varphi\) might not be \(C^1\) (in fact, we will show that it is not) and hence we cannot appeal to Theorem \ref{thm:nodal} to show that \(\varphi\) is strictly increasing on \((-P/2,0)\). 
We will instead use the double symmetrisation formula
\begin{equation}\label{eq:double symmetrisation I}
\begin{aligned}
&(L\varphi)(x + h) - (L\varphi)(x - h)\\
&\quad= \int_{-P/2}^0 (K_P(y-x) - K_P(y+x))(\varphi(y+h) - \varphi(y-h))\,\diff y
\end{aligned}
\end{equation}
to prove the strict monotonicity. It will then follow from 
Theorem \ref{thm:regularity I} that \(\varphi\) is smooth away from \(x = kP\), \(k \in \Z\).
The validity of the formula \eqref{eq:double symmetrisation I} follows from the evenness and periodicity of \(K_P\) and \(\varphi\).
Note that both factors in the integrand are nonnegative for \(x\in (-P/2, 0)\) and \(h\in (0, P/2)\).
We also have the equality
\begin{equation}\label{eq:smoothregularity}
(\mu - \varphi(x) - \varphi(y)) (\varphi(x) - \varphi(y)) = L\varphi(x) - L\varphi(y),
\end{equation}
which shows that \(L\varphi(x)=L\varphi(y)\) whenever \(\varphi(x)=\varphi(y)\).
This identity, together with \eqref{eq:double symmetrisation I}, yields that \(\varphi\) is strictly increasing on \((-P/2,0)\) (recall that \(\varphi\) is nonconstant by assumption).
The differentiation under the integral sign for \(x \in (-P/2,0)\) in \eqref{eq:essentialapproach}
can now be justified by applying Fatou's lemma to \((\frac{\mu}{2} - \varphi(x))\varphi^\prime(x) = \lim_{h \to 0} (L\varphi(x+h) - L\varphi(x-h))/4h\).
From \eqref{eq:double symmetrisation I}, we then obtain that
\[
({\textstyle \frac{\mu}{2}}-\varphi(x))\varphi^\prime(x)\ge 
\frac12\int_{-P/2}^0 (K_P(y-x) - K_P(y+x))\varphi^\prime(y)\,\diff y.
\] 
The rest of the proof remains unchanged.
\end{proof}

\begin{theorem}[Regularity II]\label{thm:regularity II}
Let \(\varphi \leq \frac{\mu}{2}\) be a solution of the steady Whitham equation \eqref{eq:steadywhitham}, which is even, nonconstant, and nondecreasing on \((-P/2,0)\) with \(\varphi(0) = \frac{\mu}{2}\). Then:
\begin{itemize}
\item[(i)] \(\varphi\) is smooth on \((-P, 0)\).
\item[(ii)]  \(\varphi \in C^{1/2}(\R)\).
\item[(iii)] \(\varphi\) has H\"older regularity precisely \(\frac12\) at \(x=0\), that is, there exist constants $0<c_1<c_2$ such that
\begin{equation}
\label{eq:precise regularity at origin}
c_1|x|^{\frac12}\le {\textstyle \frac{\mu}{2}}-\varphi(x)\le c_2|x|^{\frac12}
\end{equation}
for $|x|\ll 1$.
\end{itemize}
\end{theorem}

\begin{remark}
Note that the period \(P\) in Theorem~\ref{thm:regularity II} could be infinite (see Remark \ref{remark:essentialbound}). This is also the reason why we use  \(K\), and not its periodisation \(K_P\), in the proof --- to get a uniform argument for the periodic and potential solitary case.
\end{remark}

\begin{proof}
Part (i) follows directly from Theorem \ref{thm:regularity I} (iii) since $\varphi$ is strictly increasing on $(-P/2,0)$.

We next show that $\varphi \in C^{\alpha}(\R)$ for all \(\alpha < \frac{1}{2}\). Recall first that \(L\) maps \(\mathcal{C}^0(\R)\) continuously into \(\mathcal{C}^{1/2}(\R) = C^{1/2}(\R)\), see Section~\ref{sec:Kp}.
The equality \eqref{eq:smoothregularity} implies that at any point where
\(\varphi(x) < \frac{\mu}{2}\), the functions \(\varphi\) and \(L\varphi\) have the same H\"older regularity (this provides an immediate proof of that \(\varphi\) is at least \(C^\frac{1}{2}\) wherever \(\varphi(x) \neq \frac{\mu}{2}\)). At any point \(x_0\) where \(\varphi(x_0) = \frac{\mu}{2}\), \eqref{eq:smoothregularity} reduces to
\begin{equation}
(\varphi(x_0) - \varphi(x))^{2} = L\varphi(x_0) - L\varphi(x).
\end{equation}
This means that if \(L\varphi\) is \(2\alpha\)-H\"older continuous at \(x_0\), then \(\varphi\) is \(\alpha\)-H\"older continuous at the same point. So say that \(\varphi \in C^{\alpha}(\R)\) (here one needs the uniformity in \(x\)). Then \(L\varphi \in C^{\alpha + 1/2}(\R)\) and \(\varphi\) has H\"older regularity \(\frac{1}{2}(\alpha + \frac{1}{2})\) at \(x_0\). In view of that \(\frac{1}{2}(\alpha + \frac{1}{2}) > \alpha\) for \(\alpha < \frac{1}{2}\), this shows that for any such \(\alpha\), the function \(\varphi\) has the corresponding H\"older regularity at \(x=0\).

This argument can be extended to a global one in the following way. Since \(\varphi \leq \frac{\mu}{2}\), we have \(\varphi(x) - \varphi(y) \leq \mu - \varphi(x) - \varphi(y)\), so \eqref{eq:smoothregularity} shows that
\[
(\varphi(x) - \varphi(y))^2 \leq |L\varphi(x) - L\varphi(y)|,
\]
for all \(x, y \in \R\). Thus \(\varphi \in C^{\alpha}(\R)\) for all \(\alpha < \frac{1}{2}\). 

We next prove that $\varphi \in C^{\frac{1}{2}}(\R)$. The first part of the argument concerns the \(C^{1/2}\)-estimate \eqref{eq:precise regularity at origin} at the point \(x=0\); the second the corresponding global estimate. Part (iii) in Theorem~\ref{thm:regularity II} then follows from the first estimate combined with the choice with \(x = x_0\) in Lemma~\ref{lemma:essentialbound}, which proves the lower bound in \eqref{eq:precise regularity at origin}. To start with, let 
\[
u(x):=\frac{\mu}{2}-\varphi(x)=\varphi(0)-\varphi(x).
\] 
We want to show that there is a constant \(c_2>0\) such that  \(|u(x)|\le c_2|x|^{1/2}\) for all \(x\). Note first that \(u\) satisfies the equation
\allowdisplaybreaks
\begin{equation}
\label{eq:simple symmetrisation}
\begin{aligned}
(u(x))^2&=(L\varphi)(0)-(L\varphi)(x)\\
&=\frac12\int_{\R} (K(x+y)+K(x-y)-2K(y))u(y)\, \diff y.
\end{aligned}
\end{equation}
We claim that there is a constant \(c_2\), independent of \(\alpha\), such that
\begin{equation}
\label{eq:key estimate}
\frac12\int_{\R} |K(x+y)+K(x-y)-2K(y)|(w(y))^{\alpha}\, \diff y\le  c_2 (w(x))^{2\alpha},\quad 0\le \alpha \le 1/2,
\end{equation}
where
\[
w(x)=\min\{|x|,1\}.
\]
Indeed, for \(|x|\ge 1\), this follows directly from the integrability of \(K\) and the fact that \(\|w\|_\infty \le 1\).
For \(|x|\le 1\), we use the splitting 
\[
K(x)=\frac{1}{\sqrt{2\pi|x|}}+K_\text{reg}(x)
\]
from Proposition \ref{prop:kernel decomposition}.
For the regular part, we note that
\begin{align*}
&\int_{\R} |K_\text{reg}(x+y)+K_\text{reg}(x-y)-2K_\text{reg}(y)|(w(y))^{\alpha}\, \diff y\\
&\le \int_{\R} |K_\text{reg}(x+y)+K_\text{reg}(y-x)-2K_\text{reg}(y)|\, \diff y\\
&\lesssim \int_{\R} \frac{|x|^2}{(1+|y|)^{5/2}}\, \diff y\\
&\lesssim |x|^2,
\end{align*} 
for \(|x|\le 1\),  where we have used Taylor expansion around \(y\) and, from Proposition~\ref{prop:kernel decomposition}, the estimate
\[
|K_\text{reg}^{\prime\prime}(y)|=\left|K^{\prime\prime}(y)-\frac{3}{4\sqrt{2\pi} |y|^{5/2}}\right|
\lesssim \frac1{(1+|y|)^{5/2}}
\]
in the third line (recall that \(K_\text{reg}\) is smooth and that \(K^{\prime\prime}\) decays exponentially).
For the singular part, we use the identity
\begin{align*}
&\int_{\R} \left|\frac1{\sqrt{|x+y|}}+\frac1{\sqrt{|y-x|}}-\frac{2}{\sqrt{|y|}}\right||y|^{\alpha}\, \diff y\\
&= |x|^{\frac12+\alpha}\int_{\R}\left|\frac1{\sqrt{|1+s|}}+\frac1{\sqrt{|s-1|}}-\frac{2}{\sqrt{|s|}}\right||s|^{\alpha}\, \diff s,
\end{align*}
where \(y=xs\) and the integral converges since 
\[
\left|\frac1{\sqrt{|1+s|}}+\frac1{\sqrt{|s-1|}}-\frac{2}{\sqrt{|s|}}\right| \lesssim  |s|^{-\frac{5}{2}}, \quad |s|\gg 1.
\]
The estimate \eqref{eq:key estimate} now follows by noting that
\(|x|^{\frac12+\alpha}\le |x|^{ 2\alpha}\) for \(|x|\le 1\) and \(0\le \alpha\le 1/2\). 
Combining \eqref{eq:simple symmetrisation} with \eqref{eq:key estimate}, we obtain that
\[
\|w^{-\alpha} u\|_\infty^2 \le c_2\|w^{-\alpha} u\|_\infty.
\]
For \(\alpha<1/2\) we know a priori that the right-hand side is bounded. 
Hence, we obtain that
\[
\|w^{-\alpha} u\|_\infty\le c_2
\]
and thus
\[
|u(x)|\le c_2|x|^{\alpha}
\]
for all \(\alpha \in [0,1/2)\) and \(|x|\le 1\). Letting \(\alpha \to 1/2\) shows that
\[
|u(x)|\le c_2|x|^{\frac12}
\]
for all \(|x|\le 1\).
We have thus proved the upper bound in \eqref{eq:precise regularity at origin}.

To establish global \(C^{1/2}\)-H\"older regularity (that is, to prove (ii)), we shall use a second double symmetrisation formula, 
\begin{equation}\label{eq:double symmetrisation II}
\begin{aligned}
&(L\varphi)(x + h) - (L\varphi)(x - h)\\
&\quad= \int_{-P/2}^0 (K_P(y+h) - K_P(y-h))(\varphi(y-x) - \varphi(y+x))\,\diff y,
\end{aligned}
\end{equation}
which follows in the same way as \eqref{eq:double symmetrisation I}.  Equivalently, \eqref{eq:double symmetrisation II} reads
\begin{align*}
&\left(\mu - \varphi(x+h) - \varphi(x-h)\right)\left(\varphi(x+h) - \varphi(x-h)\right)\\ 
&\quad= \int_{-\infty}^0 (K(y+h) - K(y-h))(\varphi(y-x) - \varphi(y+x))\,\diff y,
\end{align*}
where we consider \(x \in (-P/2,0)\) and \( 0 <  h \leq |x| \leq \delta\) (all factors are symmetric in \(x\) and \(h\), so we may rename the smallest of them \(h\); \eqref{eq:sqrt h} implies that there is no loss of generality in this choice). Note that \(\varphi\) is continuously differentiable on any set \((\delta,P)\), so it is sufficient to establish that
\begin{equation}\label{eq:sqrt h}
\sup_{0 < h \leq |x| \leq \delta} \frac{\varphi(x+h) - \varphi(x-h)}{\sqrt{h}} < \infty,
\end{equation}
for some \(\delta \ll 1\). First, note that
\[
\mu - \varphi(x+h) - \varphi(x-h) \geq \frac{\mu}{2} - \varphi(x-h) \geq  \frac{\mu}{2} - \varphi(x),
\]
whence \(\varphi(x+h) - \varphi(x-h) \geq 0\) implies that
\begin{equation}\label{eq:fix estimated doublesym}
\begin{aligned}
&\left(\frac{\mu}{2} - \varphi(x)\right)\left(\varphi(x+h) - \varphi(x-h)\right)\\ 
&\quad\leq \int_{-\infty}^0 (K(y+h) - K(y-h))(\varphi(y-x) - \varphi(y+x))\,\diff y.
\end{aligned}
\end{equation}
We shall interpolate between two estimates for \(\varphi(y-x) - \varphi(y+x)\), namely
\begin{align*}
|\varphi(y-x) - \varphi(y+x)| &\lesssim \|\varphi\|_{C^\alpha} \min (|x|^\alpha, |y|^\alpha), \qquad 0 < \alpha < 1/2,\\
\intertext{and}
|\varphi(y-x) - \varphi(y+x)| &\lesssim |\varphi|_{C^{1/2}_{0}} \max (|x|^\frac{1}{2}, |y|^\frac{1}{2}),
\end{align*}
where the second follows from the already proved (upper and lower)  estimate  \(\frac{\mu}{2} - \varphi(x) \sim |x|^{1/2}\). Thus
\[
|\varphi(y-x) - \varphi(y+x)| \lesssim \|\varphi\|_{C^\alpha}^\eta \min (|x|^{\alpha\eta}, |y|^{\alpha\eta}) \max (|x|^{\frac{1-\eta}{2}}, |y|^{\frac{1-\eta}{2}}),
\]
for all \((\alpha,\eta)  \in (0,\frac{1}{2}) \times [0,1]\). We now choose \(\eta\) such that 
\[
\alpha\eta = \frac{1-\eta}{2}, \quad \text{ meaning that }\quad \eta = \frac{1}{1+2\alpha} \in (1/2,1).
\] 
Then
\[
\chi_\eta(x,y) := |\varphi(y-x) - \varphi(y+x)| \lesssim \|\varphi\|_{C^\alpha}^\eta |xy|^{\alpha\eta} = \|\varphi\|_{C^\alpha}^\eta |xy|^{\frac{1-\eta}{2}},
\]
and, consequently,
\begin{equation}\label{eq:interpolation argument}
\begin{aligned}
&\int_{-\infty}^0 (K(y+h) - K(y-h))(\varphi(y-x) - \varphi(y+x))\,\diff y\\ 
&\quad  \lesssim \|\varphi\|_{C^\alpha}^\eta \bigg( |x|^{\alpha\eta}\frac1{\sqrt{2\pi}} \int_{-\infty}^0  \left(\frac{1}{\sqrt{|y+h|}} - \frac{1}{\sqrt{|y-h|}}  \right) |y|^{\alpha\eta}\,\diff y\\ 
&\qquad + \int_{-\infty}^0 |K_\text{reg}(y+h) - K_\text{reg}(y-h)| \chi_\eta(x,y)\,\diff y\bigg)\\
&\quad= \|\varphi\|_{C^\alpha}^\eta \bigg( |x|^{\alpha\eta} \underbrace{|h|^{1/2}|h|^{\alpha\eta}}_{|h|^{1- \frac{\eta}{2}}}\frac1{\sqrt{2\pi}} \int_{-\infty}^0  \left(\frac{1}{\sqrt{|s+1|}} - \frac{1}{\sqrt{|s-1|}}  \right) |s|^{\alpha\eta}\,\diff s + O(h)\bigg).
\end{aligned}
\end{equation}
Here we have used the smoothness and decay of \[K_\text{reg}^\prime = K^\prime + \frac{1}{2\sqrt{2\pi}} \sign{(\cdot)} |\cdot|^{-3/2}\] to estimate the regular part:
\begin{align*}
&\int_{-\infty}^0 |K_\text{reg}(y+h) - K_\text{reg}(y-h)| \chi_\eta(x,y)\,\diff y\\
&\leq 2h \| \varphi\|_\infty \int_{-\infty}^0  \int_{-1}^1 |K_\text{reg}^\prime(y +th)| \,\diff t \,\diff y\\
&\lesssim h,
\end{align*}
since \(K^\prime\) has exponential decay and \(|\cdot|^{-3/2}\) is integrable at infinity. 
Note that the factor \(|s|^{\alpha\eta}\) in \eqref{eq:interpolation argument} satisfies \(\alpha\eta \leq 1/4\) by choice of \(\eta\), so that the integral is uniformly bounded for all \(\alpha \in (0,\frac{1}{2})\). Combining \eqref{eq:interpolation argument} with \eqref{eq:fix estimated doublesym}, one therefore obtains
\[
\left(\frac{\frac{\mu}{2} - \varphi(x)}{|x|^{\frac{1-\eta}{2}}}\right)\left(\frac{\varphi(x+h) - \varphi(x-h)}{h^{1-\frac{\eta}{2}}}\right) \lesssim \|\varphi\|_{C^\alpha}^\eta. 
\]
Now, in view of that \(|x| \geq h\) and \(\frac{\mu}{2} - \varphi(x) \gtrsim |x|^{1/2}\), one may further reduce this estimate to
\[
\frac{\varphi(x+h) - \varphi(x-h)}{h^{1-\eta}}\lesssim \|\varphi\|_{C^\alpha}^\eta,
\]
and, because \(1-\eta = \alpha\eta + \frac{1-\eta}{2}\) we obtain that
\[
\left(\frac{\varphi(x+h) - \varphi(x-h)}{h^{\alpha}}\right)^\eta \left(\frac{\varphi(x+h) - \varphi(x-h)}{h^{1/2}}\right)^{1-\eta} \lesssim \|\varphi\|_{C^\alpha}^\eta.
\]
Since \(h \ll 1\), we can estimate \(h^{-1/2}\) from below with \(h^{-\alpha}\). Note that 
\begin{equation}\label{eq:alpha-norm bounded by local estimate}
\|\varphi\|_{C^\alpha} \lesssim \max \left\{1, \sup_{0 < h < |x| < \delta} \frac{|\varphi(x+h) - \varphi(x-h)|}{h^{\alpha}} \right\}
\end{equation}
for all \(\alpha \leq 1/2\). Indeed,  \(\varphi(x+y) - \varphi(x-y)\) is symmetric in \(x\) and \(y\), so the largest difference quotient is always obtained by dividing with the smallest of \(|x|\) and \(|y|\), whence it is enough to consider \(0 < |y| \leq |x| \leq P/2\). If \(|y| \geq \delta/2\), then 
\[
\frac{|\varphi(x+y) - \varphi(x-y)|}{|y|^{\alpha}} \leq \frac{4 \|\varphi\|_{\infty}}{\delta},
\] 
for all \(x\). If, on the other hand, \(|x| \geq \delta\) and \(|y| \leq \delta/2\), then
\[
\frac{|\varphi(x+y) - \varphi(x-y)|}{|y|^{\alpha}} \leq  2\left(\frac{\delta}{2}\right)^{1-\alpha} \|\varphi\|_{C^1([\delta/2,P/2])}
 \lesssim \|\varphi\|_{C^1([\delta/2,P/2])},
\]
by the mean value theorem. For a given \(\varphi\) and \(\delta\), both these quantities are \(\bigOh(1)\), and independent of \(\alpha \in (0,\frac{1}{2})\). Hence \eqref{eq:alpha-norm bounded by local estimate} holds and, in any case, we obtain that
\[
\sup_{0 < h < |x| < \delta} \left(\frac{\varphi(x+h) - \varphi(x-h)}{h^{\alpha}}\right)^{1-\eta} \lesssim 1, \qquad 1- \eta \in \left[{\textstyle\frac{1}{3},\frac{1}{2}}\right),
\]
uniformly for \(\alpha \in [\frac{1}{4},\frac{1}{2})\). As above, the uniformity in \(\alpha\) allows for letting \(\alpha \to \frac{1}{2}\) to obtain the global \(C^{1/2}\)-regularity of \(\varphi\).
\end{proof}




\section{Global bifurcation and the Whitham conjecture}\label{sec:global}
We now fix \(\alpha \in (\frac{1}{2},1)\), and consider $C^{\alpha}_{\text{even}}(\s_P)$, the space of even and 
$\alpha$-H{\"o}lder continuous real-valued functions on the circle $\s_P$ of finite circumference \(P > 0\).  Let furthermore  \(F  \colon C^\alpha_\text{even}(\s_P) \times \R \to C^{\alpha}_\text{even}(\s_P)$ be the operator defined by
\begin{align}\label{eq:fop_whitham}
F(\varphi,\mu) & = \mu \varphi -  L \varphi -  \varphi^2 ,
\end{align}
The following local bifurcation result is an extension of results proved in \cite{EK11} (for \(P=2\pi\)) and  \cite{EK08} (for a general \(P\), but with less information on the bifurcation branches).

\begin{theorem}
\label{thm:local_whit}
$ $\\[6pt]
{\bf (i) Sub- and supercritical bifurcation.} For each finite period \(P > 0\)  and each integer \(k\ge 1\)
there exist  \(\mu_{P,k}^* = \left({\tanh(\frac{2\pi k}{P})}/(\frac{2\pi k }{P})\right)^{1/2}\)
and a local, analytic curve 
\[
{s} \mapsto (\varphi_{P,k}({s}),\mu_{P,k}({s})) \in C^\alpha_{\rm{even}}(\s_P) \times \R 
\]
of nontrivial \(P/k\)-periodic Whitham solutions with 
\(\Diff_{s} \varphi_{P,k}(0) = \cos(2\pi k \cdot /P)\) that bifurcates from the 
trivial solution curve \(\mu \mapsto (0,\mu)$ at $(\varphi_{P,k}(0),\mu_{P,k}(0)) = (0,\mu_{P,k}^*)\). The curve can be parametrised in such a way that \[
s \mapsto \mu_{P,k}({s}) \text{ is even,} 
\]
and there exist positive numbers  \(P_1  < P_2\) with the property that 
\[
\mu_{P< kP_1,k}''(0)>0, \quad \mu_{P> kP_2,k}''(0)<0.
\]
Hence, a subcritical pitchfork bifurcation takes place at \((0,\mu_{P,k}^*)\) for \(P>kP_2\), while a supercritical pitchfork bifurcation occurs for \(P < kP_1\).
 \\[6pt] 
{\bf (ii) Transcritical bifurcation.}
At $\mu =1$ the trivial solution curve $\mu \mapsto (0,\mu)$ intersects the curve $\mu \mapsto (\mu -1, \mu)$ of constant solutions $\varphi_0 = \mu - 1$.\\[6pt] 
Together, the solutions in (i) and (ii) constitute all nonzero solutions of $F(\varphi,\mu) = 0$  in $C^\alpha_{\rm{even}} (\s_P) \times \R$ in a neighbourhood of the trivial solution curve \(\{(0,\mu)\colon \mu \in \R\}\).
\end{theorem}

\begin{remark}
It is only the quotient between \(k\) and \(P\) that is relevant in the statement of this theorem. Two solution branches  with the same quotient coincide locally near the bifurcation point. However, global continuations of such curves could differ (e.g.~due to subharmonic bifurcations). Moreover, the distinction between the branches will be useful in the proof of Theorem \ref{thm:noperiod}, where we wish to keep \(P\) fixed. 
\end{remark}

\begin{remark}\label{rem:mu2}
Plotting the function \(\mu''(0)\), one can see that there is a positive number  \(P_0 \approx 2.57\) with the more precise property that 
\[
\mu_{P< kP_0,k}''(0)>0, \quad \mu_{P> kP_0,k}''(0)<0, \quad \text{ while }\quad \mu_{kP_0,k}^{(4)}(0)>0.
\]
Hence, a subcritical pitchfork bifurcation takes place at \((0,\mu_{P,k}^*)\) for \(P>kP_0\), while a supercritical pitchfork bifurcation occurs for \(P\le kP_0\). In spite of the explicit expression of \(\mu''(0)\) given in \eqref{eq:mu2}, we have not been able to establish this analytically. The analytical calculations leading up to the above (numerical) results are included in the proof of Theorem~\ref{thm:local_whit}.
\end{remark}

\begin{proof}
The proof makes use of the same arguments as in \cite{EK11}. The only modification is that we consider a general period and obtain some additional information on the curves \((\varphi_{P,k}({s}),\mu_{P,k}({s}))\) in the case of sub- or supercritical bifurcation. In that case, the sign of \(\mu_{2\pi,1}''(0)\) was computed in \cite[Theorem 4.6]{EK11}. We repeat this computation for a general period, using a slightly different method. It suffices to consider \(k=1\); the general result follows by rescaling \(P\).
To simplify the notation, we will abbreviate \((\varphi_{P, 1}({s}), \mu_{P,1}({s}))\) by \((\varphi(s), \mu(s))\). 

We begin by showing that
\begin{equation}
\label{eq:mu even}
\mu(s)=\mu(-s)
\end{equation}
after a suitable choice of parametrisation.
We denote by
\[
[\varphi]_j=
\frac{2}{P} \int_{-P/2}^{P/2} \varphi(x) \cos\left(\frac{2\pi j x}{P}\right)\, \diff x, \quad j=0,1,2,\ldots
\]
the coefficients in the cosine expansion of an even \(P\)-periodic function \(\varphi= \frac{[\varphi]_0}{2} +
\sum_{j=1}^\infty [\varphi]_j \cos(\frac{2\pi j \cdot}{P})\).\footnote{Note that we use a slightly different convention here  compared  to the treatment of Fourier series in Section \ref{sec:Kp}.} 
We parametrise the local bifurcation curve in such a way that \([ \varphi(s) ]_1={s}\).
This parametrisation corresponds to the Lyapunov-Schmidt reduction used in \cite[Section 4.1]{EK11}.
Note that for a given even \(P\)-periodic solution \((\varphi, \mu)\), \((\varphi(\cdot+P/2),\mu)\) is also an even $P$-periodic solution and satisfies
\[
[\varphi(\cdot+P/2)]_1=-[\varphi]_1.
\]
Since 
\( [\varphi(s)(\cdot +P/2)]_1=-[\varphi(s)]_1=-{s}\), it follows by uniqueness that
\[
(\varphi(s)(\cdot +P/2), \mu(s))=(\varphi(-s), \mu(-s)).
\]
This proves \eqref{eq:mu even}.

In view of \eqref{eq:mu even} and the analyticity of $\mu^{s}$, we can write
\[
 \mu(s)=\sum_{n=0}^\infty \mu_{2n} {s}^{2n}
\]
with uniform convergence for ${s}$ in a neighbourhood of the origin.
We also expand
\[
 \varphi(s) = \sum_{n=1}^\infty \varphi_n {s}^n
\]
with convergence in \(C^\alpha_\text{even}(\s_P)\).
By uniqueness, we can compute the coefficients by substituting the above expansions into the Whitham equation and identifying terms of equal order in ${s}$.
This yields
\begin{align}
\label{eq:expansion 1}
L\varphi_1-\mu_0 \varphi_1&=0,\\
\label{eq:expansion 2}
L\varphi_2-\mu_0 \varphi_2&=-\varphi_1^2,\\
\label{eq:expansion 3}
L\varphi_3-\mu_0\varphi_3&=\mu_2 \varphi_1-2\varphi_1\varphi_2,\\
\label{eq:expansion 4}
L\varphi_4-\mu_0\varphi_4&=\mu_2\varphi_2-2\varphi_1\varphi_3-\varphi_2^2,\\
\label{eq:expansion 5}
L\varphi_5-\mu_0\varphi_5&=\mu_2\varphi_3+\mu_4 \varphi_1
-2\varphi_1\varphi_4-2\varphi_2\varphi_3.
\end{align}
By definition, $\varphi_1(x)=\cos(\xi x)$ and $\mu_0=m(\xi)$, where $\xi=2\pi/P$, so that \eqref{eq:expansion 1} is satisfied. The remaining coefficients in the power series for $\mu$ can be determined by the requirement that each right-hand side must lie in the range of the linear operator defined by the left-hand side. The functions $\varphi_n$ are then obtained by solving the resulting equations. By choice of parametrisation, $[\varphi_n]_1=0$ for each $n\ge 2$. 
Using the formula for $\varphi_1$, the right-hand side of \eqref{eq:expansion 2} reduces to
\[
-\frac12-\frac12 \cos(2\xi x),
\]
which in turn yields
\[
\varphi_2(x)=-\frac{1}{2(m(0)-m(\xi))}-\frac{1}{2(m(2\xi)-m(\xi))}\cos(2\xi x).
\]
The right-hand side of \eqref{eq:expansion 3} then simplifies to
\begin{equation}\label{eq:mu2}
\begin{aligned}
&\left(\mu_2-\frac{1}{m(\xi)-m(0)}-\frac{1}{2(m(\xi)-m(2\xi))} \right)\cos(\xi x)\\
&\quad
-\frac{1}{2(m(\xi)-m(2\xi))} \cos(3\xi x),
\end{aligned}
\end{equation}
yielding the relation
\[
\mu_2=\frac{1}{m(\xi)-m(0)}+\frac{1}{2(m(\xi)-m(2\xi))}.
\]
It is not hard to see that this function is negative as \(\xi \searrow 0\), and positive as \(\xi \nearrow \infty\). This yields the existence of the numbers \(P_1\) and \(P_2\) in the theorem. 

We now give the calculations for the higher-order derivative \(\mu^{(4)}(0)\) mentioned in Remark~\ref{rem:mu2}. By solving \eqref{eq:expansion 3}, we obtain
\begin{align*}
\varphi_3(x)=\frac{1}{2(m(2\xi)-m(\xi))(m(3\xi)-m(\xi))}\cos(3\xi x).
\end{align*}
The right-hand side of \eqref{eq:expansion 4} now reduces to
\begin{align*}
&\frac{1}{4(m(0)-m(\xi))^2}-\frac{1}{4(m(0)-m(\xi))(m(\xi)-m(2\xi))}-\frac{1}{8(m(2\xi)-m(\xi))^2}\\
&
-\frac{1}{2(m(2\xi)-m(\xi))}\left(\frac{1}{m(3\xi)-m(\xi)}{-\frac{1}{2(m(2\xi)-m(\xi))}}\right)\cos(2\xi x)
\\
&-\frac{1}{2(m(2\xi)-m(\xi))}\left(\frac{1}{m(3\xi)-m(\xi)}+\frac{1}{4(m(2\xi)-m(\xi))}\right)\cos(4\xi x),
\end{align*}
which gives
\begin{align*}
\varphi_4(x) &=
\frac{1}{4(m(0)-m(\xi))^3}
-\frac{1}{4(m(0)-m(\xi))^2(m(\xi)-m(2\xi))}
\\
&\quad-\frac{1}{8(m(0)-m(\xi))(m(2\xi)-m(\xi))^2}\\
&\quad-\frac{1}{2(m(2\xi)-m(\xi))^2}\left(\frac{1}{m(3\xi)-m(\xi)}{-\frac{1}{2(m(2\xi)-m(\xi))}}\right)\cos(2\xi x)
\\
&\quad- \frac{1}{2(m(2\xi)-m(\xi))(m(4\xi)-m(\xi))} \\
&\qquad \times  \left(\frac{1}{m(3\xi)-m(\xi)}+\frac{1}{4(m(2\xi)-m(\xi))}\right) \cos(4\xi x).
\end{align*}
In order to determine $\mu_4$, we finally compute the $\cos(\xi x)$ component of the right-hand side of \eqref{eq:expansion 5}.
This results in
\begin{align*}
\mu_4&=
\frac{1}{2(m(0)-m(\xi))^2}\left(\frac{1}{m(0)-m(\xi)}+\frac{1}{m(2\xi)-m(\xi)}\right)\\
&\quad-\frac{1}{4(m(2\xi)-m(\xi))^2} \left( \frac{1}{m(0)-m(\xi)} + \frac{3}{m(3\xi)-m(\xi)}\right)\\
&\quad +\frac{1}{4(m(2\xi)-m(\xi))^3}
\end{align*}
and one finds, numerically, that \(\mu_4>0\) for \(\xi=\xi_0\) (see Remark~\ref{rem:mu2}).
\end{proof}

With 
\[
U = \left\{ (\varphi,\mu) \in C^\alpha_\text{even}(\s_P) \times \R \colon \varphi  < \mu/2 \right\},
\]
we let
\begin{equation*}\label{eq:S}
S = \left\{ (\varphi,\mu) \in  U \colon  F(\varphi,\mu) = 0  \right\}
\end{equation*}
be our set of solutions. Note that for nonconstant solutions satisfying \(\varphi \leq \mu/2\), the wave speed \(\mu\) is a priori bounded from above. Since \(K\) is positive with \(\int_\R K(x)\,\diff x = 1\), one namely has \(\mu \sup \varphi \leq (\sup \varphi)^2 + \sup \varphi\). Because  furthermore \(\sup \varphi > 0\)  for nonconstant solutions by Lemma~\ref{lemma:apriori_1} and Remark~\ref{rem:sign-changing}, one obtains \(\mu \leq 1 + \frac{\mu}{2}\), and thus 
\begin{equation}\label{eq:mu leq 2}
\mu \leq 2.
\end{equation}
We shall later improve this general bound in the case when \(\mu = \mu(s)\) is taken along our bifurcation curve (cf.~the proof of Theorem~\ref{thm:noperiod}), but we first extend the curve to a global one. The following theorem is is an easy adaption of \cite[Theorem 4.4.]{EK11} with \(U\) and \(S\) as above.

\begin{figure}
\begin{center}
\includegraphics[width=0.5\linewidth]{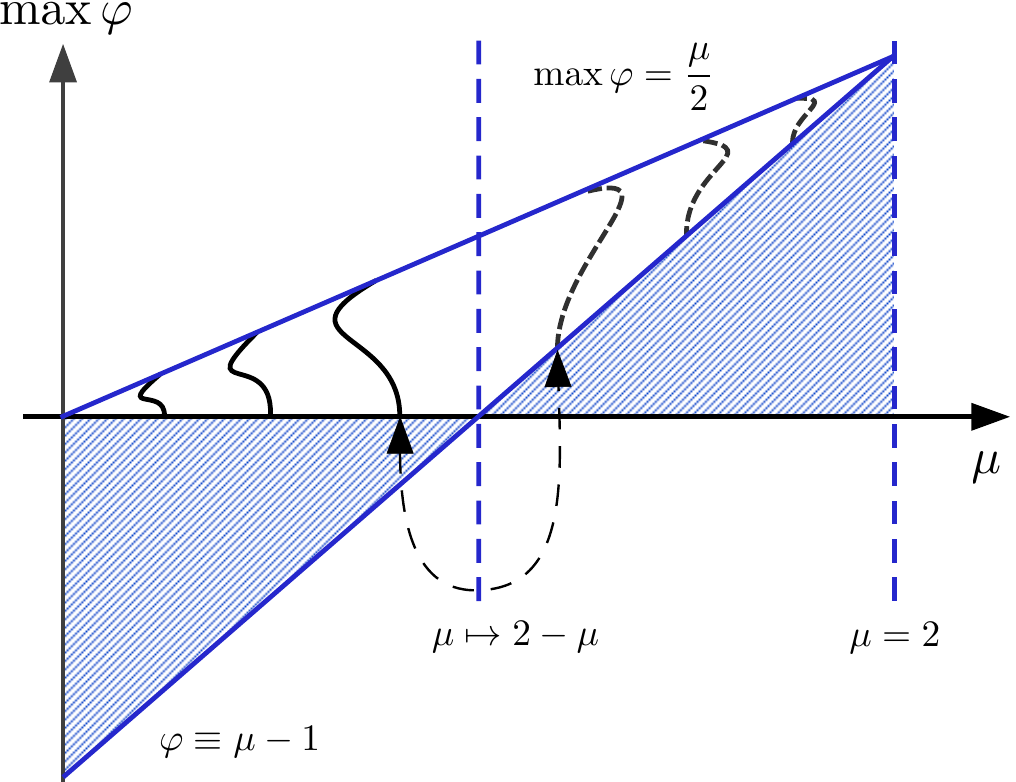}
\caption{The global bifurcation diagram obtained in Theorem~\ref{thm:global_whit}. As follows by construction and from Section~\ref{sec:nodal}, the maxima of these solutions for \(\mu \in [0,1]\) lie in the interval \([0,\mu/2]\). The solutions with wave speed \(\mu \in [1,2]\) are in one-to-one-correspondence with the former via the Galilean transformation \eqref{eq:galilean}. Along the main bifurcation branch the wave speed \(\mu\) is bounded away both from vanishing, cf. Corollary~\ref{cor:essentialbound}, and from unit speed, cf. Remark~\ref{rem:wavespeed bounded away from 1}. As proved in Theorem~\ref{thm:local_whit}, for periods \(P\geq P_0 \approx 2.57\) the bifurcation is of sub-critical pitchfork type, and in this case numerical calculations \cite{EK11,MR3390078} show a turning point near the highest wave.}
\label{fig:bifurcation}
\end{center}
\end{figure}

\begin{theorem}[Global bifurcation]
\label{thm:global_whit}
Whenever \(\mu_{P,1}^{(j)}(0) \neq 0\) for some \(j \in \N\) in Theorem~\ref{thm:local_whit}, the curves \(s \mapsto (\varphi_{P,1}({s}),\mu_{P,1}({s}))\) of solutions to the Whitham equation extend to global continuous curves of solutions $\mathfrak{R}_{P} \colon \R_{\geq 0} \to S$, that allow a local real-analytic  reparametrisation around each \({s} > 0\). One of the following alternatives holds:
\begin{itemize}
\item[(i)] $\| (\varphi_{P,1}({s}), \mu_{P,1}({s}))  \|_{C^\alpha(\s) \times \R} \to \infty$  as ${s} \to \infty$.
\item[(ii)] \(\dist(\mathfrak{R}_P, \partial U) = 0.\)
\item[(iii)] ${s} \mapsto  (\varphi_{P,1}({s}),\mu_{P,1}({s})) $ is (finitely) periodic.
\end{itemize}
\end{theorem}

\begin{remark}
If \(\mathfrak{R}_P({s}_1) = \mathfrak{R}_P({s}_2)\) for some \({s}_1 \neq {s}_2\) at a point where \( \ker \Diff F (\mathfrak{R}({s}_1)) = \{0\}\), then alternative (iii) in Theorem~\ref{thm:global_whit} occurs with \(|{s}_1 - {s}_2|\) being a multiple of the period. Also, the values of \({s}\) for which the kernel of \(\Diff F(\mathfrak{R}({s}))\) is nontrivial are isolated. For both these facts, see \cite{MR1956130}.
\end{remark}

\begin{remark}
Concerning the assumption that \(\mu_{P,1}^{(j)}(0)\) should be nonzero for some \(j \in \N\), see also the discussion in Remark~\ref{rem:mu2}.
\end{remark}

\begin{proof}
This assertion was proved in the case \(P=2\pi\) in \cite{EK11} using compactness properties of the operator \(L\) and the fact that \(  \mu_{2\pi,1}^{\prime\prime}(0)  \neq 0\). For a general period \(P>0\) the assertion follows in the same way using the assumption \(\mu_{P,1}^{(j)}(0) \neq 0\).
\end{proof}

We shall now prove that alternative (iii) in Theorem~\ref{thm:global_whit} is excluded, and that (i) and (ii) happen simultaneously as \({s} \to \infty\) along the primary bifurcation branch \(\mathfrak{R}_P\). To that aim, let \(\mu^* = \mu_{P,1}^*\) be the primary bifurcation point from Theorem~\ref{thm:local_whit} and let
\[
\varphi^* = \cos(2\pi \cdot /P)
\]
be the direction of bifurcation in \(C^{\alpha}(\s_P)\). We  follow the route of \cite{EK11}, adding more information to the behaviour along the bifurcation branch. 

We start by proving that alternative (iii) in Theorem~\ref{thm:global_whit} cannot occur. To that aim, introduce 
\[
\mathcal{K} = \{ \varphi \in C_\text{even}^\alpha(\s_P) \colon \varphi \text{ is nondecreasing on } (-P/2,0)\},
\]
which is a closed cone in \(C^\alpha(\s_P)\). Let furthermore \(\varphi({s}) = \varphi_{P,1}({s})\), \(\mu({s}) = \mu_{P,1}({s})\),  \(\mathfrak{R} = \mathfrak{R}_{P}\), and let \(\mathfrak{R}^1\) and \(S^1\) denote the \(\varphi\)-components of \(\mathfrak{R}\) and \(S\), respectively. 

\begin{theorem}\label{thm:noperiod}
Alternative (iii) in Theorem~\ref{thm:global_whit} cannot occur.
\end{theorem}

\begin{remark}\label{rem:noperiod}
Our proof of Theorem~\ref{thm:noperiod} is based on  \cite[Theorem 9.2.2]{MR1956130}, but has been rewritten to deal with the transcritical curve of constant solutions \(\mu \mapsto (\mu -1, \mu)\) crossing the line of trivial solutions at \(\mu = 1\). Note in particular in that \(\varphi({s}) \in \mathcal{K}\setminus \{0\}\). 
\end{remark}

\begin{remark}\label{rem:wavespeed bounded away from 1}
The proof of Theorem~\ref{thm:noperiod} furthermore shows that \(\mu({s}) < 1\), uniformly for all \({s}\). We have \(\mu(s) < 1\) for small \(s\), and Proposition~\ref{prop:mean} implies that the only way to reach \(\mu = 1\) is by approaching \(\varphi = 0\). Theorem~\ref{thm:local_whit}~(ii) holds that the unique solutions in a neighbourhood of \((\varphi,\mu) = (0,1)\) are the constant solutions.  Since we show below that the main bifurcation curve does not connect to the two lines of constant solutions, it follows that the wave speed is bounded away from \(1\) from the left.
\end{remark}

\begin{proof}
If \(\varphi({s}) \in \mathcal{K}\setminus \{0\}\) for all \({s} > 0\) there is nothing to prove, so assume for a contradiction that there exists \(\bar {s}\); the largest positive number such that \(\varphi({s}) \in \mathcal{K}\setminus \{0\}\) for all \({s} < \bar {s}\). Since \(\mathcal{K}\) is closed in \(C_\text{even}^\alpha(\s_P)\), we have \(\varphi(\bar{s}) \in \mathcal{K}\). We now argue that \(\varphi(\bar{s}) = const\), by showing that if \(\varphi \in \mathfrak{R}^1 \cap  \mathcal{K}\) is nonconstant, then \(\varphi\) is an interior point of \(S^1 \cap \mathcal{K}\) with respect to the \(C^\alpha\)-metric relative to \(S^1\). 

Thus, let \(\varphi\) be a nonconstant function on the main bifurcation branch that is nondecreasing on \((-P/2,0)\). According to Theorem~\ref{thm:regularity I}, \(\varphi\) is then smooth. Hence we can apply Theorem~\ref{thm:nodal}, which shows that \( \varphi' > 0\) on \((-P/2,0)\), \(\varphi^{\prime\prime}(0) < 0\) and \(\varphi^{\prime\prime}(P/2) > 0\). Let \(\phi\) be another solution, lying within \(\delta\)-distance to \(\varphi\) in \(C^\alpha\), with \(\delta \ll 1\) small enough for \(\phi < \frac{\mu}{2}\) to hold. Then, for both these solutions, iteration of \eqref{eq:bootstrapping} yields a continuous fixed-point map \(C^\alpha  \to C^k\), \(k \geq 1\) arbitrary, so that in fact \(\|\varphi - \phi\|_{C^2} < \tilde \delta \ll 1\), where \(\tilde \delta\) can be made arbitrarily small by choosing \(\delta\) even smaller. 
It follows that, for \(\delta\) small enough, \(\phi\) is strictly increasing on \((-P/2,0)\). This means that \(\varphi(\bar {s}) = const\) (anything else would violate the definition of \(\bar{s}\)).

If \(\varphi(\bar{s}) = 0\), Theorem~\ref{thm:local_whit} enforces \(\mu(\bar{s})\) to be a bifurcation point. To exclude \(\varphi(\bar{s}) = const \neq 0\), note first that all nonzero constant solutions are of the form \(\varphi = \mu - 1\), \(\mu \neq 1\). Now, given that \(\varphi(\bar{s})\) is a nonzero constant (which we shall refute)  Proposition~\ref{prop:mean} implies that \(\mu(\bar{s}) < 1\), since in passing \(\mu = 1\) the solution \(\varphi\) would have to vanish, which in turn would imply that \(\varphi(\bar{s}) = 0\). We now claim that, for \(\mu < 1\), the curve \(\mu \mapsto (\mu -1,\mu)\) of trivial solutions is locally unique, meaning that no other solutions in \(S\) connect to this curve. The key to this observation is the Galilean transformation 
\begin{equation}\label{eq:galilean}
\mu \mapsto 2 - \mu, \qquad \varphi \mapsto \varphi + 1 - \mu, 
\end{equation}
giving a one-to-one correspondence between solutions with wave speed \(\mu < 1\) and such with \(\mu > 1\). In particular, \eqref{eq:galilean} defines a map  \((\mu -1, \mu) \mapsto (0,2-\mu)\) between the two lines of constant solutions in the \((\varphi,\mu)\)-plane (see Figure~\ref{fig:bifurcation}). But according to Theorem~\ref{thm:local_whit}, there are no nonzero solutions connecting to the line \((0,\tilde\mu)\) at \(\tilde\mu = 2 - \mu > 1\). Thus \(\varphi(\bar{s}) = 0\). 

We now rule out \(\mu(\bar{s}) = 1\) in the case \(\varphi(\bar{s})= 0\). From Proposition~\ref{prop:mean} we know that \(\mu({s}) < 1 \), for all \({s} < \bar{s}\). Also, Theorem~\ref{thm:local_whit} implies that the only solutions with \(\mu < 1\) connecting to  \((\varphi, \mu) = (0,1)\) lie on the curves \((0,\mu)\) and \((\mu -1,\mu)\) of constant solutions. The latter curve we already proved is not connected to \(\mathcal{K} \setminus \{0\}\) for \(\mu < 1\). The only remaining possibility for \(\mu(\bar{s}) = 1\) would be that \(\mathfrak{R}\) connects \((0,\mu^*)\) to  \((0,1)\) via the line of zero solutions. Since this violates the definition of \(\bar{s}\), we conclude that \(\mu(\bar{s}) = \mu_{P,k}^*\), for some \(k \geq 1\).

Thus, we may assume that the point \((\varphi,\mu)(\bar{s})\) is a  local bifurcation point, and according to Theorem~\ref{thm:global_whit} we may choose a real-analytic reparametrisation of \(\mathfrak{R}\) around that point. In view of that \(\varphi(\bar{s}) = 0\), there then exists a largest integer \(j\) such that 
\[
\varphi({s}) = \frac{\Diff^j_{s} \varphi(\bar{s})}{j!}({s} - \bar{s})^j + O(|{s} - \bar{s}|^{j+1}).
\]
By considering \({s} < \bar{s}\), one sees that 
\[
(-1)^j \Diff^j_{s}  \varphi(\bar{s}) \in  \mathcal{K} \setminus \{0\}. 
\]
On the other hand, by differentiating \(F(\varphi({s}),\mu({s})) = 0\) \(j\) times with respect to \({s}\), one obtains that
\[
(-1)^j \Diff_\varphi F(0,\mu_{P,k}^*)   \Diff^j_{s}  \varphi(\bar{s}) = 0,
\]
so that \(\phi = (-1)^j \Diff^j_{s} \varphi(\bar{s})\) fulfils \((L - \mu_{P,k}^*)  \phi = 0\). This enforces \(\phi(x) = \tau \cos(2\pi k x/P)\), since we are in a space of even \(P\)-periodic functions. Now, such functions cannot lie in \(\mathcal{K}\) if \(k \geq 2\). Also, since \(-\varphi^* \not \in \mathcal{K}\), we have found that for \({s} < \bar{s}\) but sufficiently close, \(\mathfrak{R}\) coincides with the primary branch (that is, with itself) for \(0 <{s} \ll 1\).\footnote{Note here that \(\mathfrak{R}_{{s} \ll 1}\) belongs to \(\mathcal K\): From Theorem~\ref{thm:local_whit} we get that \(\varphi({s}) = {s} \cos(2\pi k x/P) + O({s}^2)\) in \(C^\alpha(\s_P)\), and from Theorem~\ref{thm:regularity I} that all small solutions are smooth. By combining these two properties one gets the desired uniformity in \(x\) to conclude that \(\varphi({s})\) is strictly increasing on \((-P/2,0)\) for \({s} \ll 1\) small enough.} This, in turn, implies that there are countably many pairs \(({s}_{1,j},{s}_{2,j})\) with \({s}_{1,j} \searrow 0\) and \({s}_{2,j} \nearrow \bar{s}\) for which \(\mathfrak{R}({s}_{1,j}) = \mathfrak{R}({s}_{2,j})\). In light of the remark following Theorem~\ref{thm:global_whit} this is a contradiction, and we conclude that \(\bar{s}\) does not exists. Thus, \(\varphi({s}) \subset \mathcal{K} \setminus \{0\}\) for all \({s} > 0\).
\end{proof}

To exclude a trivial wave in the limit \(s \to \infty\) we need a couple of results which show that \(\mu\) is a priori bounded away from \(0\) along the bifurcation branches. Recall that we already know that \(\mu\) is bounded from above by \(1\); cf.~Remark \ref{rem:wavespeed bounded away from 1}.  

\begin{lemma}\label{lemma:convergence}
Any sequence of Whitham solutions $(\varphi_n,\mu_n) \in S$ with \((\mu_n)_n\) bounded has a subsequence which converges uniformly to a solution $\varphi$. 
\end{lemma}
\begin{proof}
We have 
\[
\|\varphi\|_\infty^2 \leq  \|\mu \varphi\|_{\infty} + \|L\|_{B(L^{\infty}(\R))} \|\varphi\|_\infty  =  ( |\mu| + 1) \|\varphi\|_\infty,
\] 
so that \((\varphi_n)_n\) is bounded whenever \((\mu_n)_n\) is. Since \(K\) is integrable and continuous almost everywhere, it follows by dominated convergence that \((L\varphi_n)_n\) is equicontinuous. Arzela--Ascoli's lemma then implies the existence of a uniformly convergent subsequence.
\end{proof}

\begin{corollary}\label{cor:essentialbound}
For any fixed period \(P> 0\), one has
\[
\mu(s) \gtrsim 1,
\]
uniformly for all \({s} \geq 0\) along the global bifurcation curve in Theorem~\ref{thm:global_whit}.
\end{corollary}
\begin{proof}
Assume for a contradiction that there is a sequence \((\mu_n)_n\) such that \(\mu_n \to 0\) as \(n \to  \infty\), while at the same time \(\varphi_n = \varphi_{\mu_n}\)  is a sequence along the global bifurcation curve in Theorem~\ref{thm:global_whit}. According to Lemma~\ref{lemma:convergence}, a subsequence \((\varphi_{n_k})_k\) converges uniformly to a solution \(\varphi_0\) of \eqref{eq:steadywhitham}. Because \(\varphi_{n_k} < \frac{\mu_k}{2} \to 0\) as \(k \to \infty\), it follows that \(\varphi_0 \leq 0\). In view of Lemma~\ref{lemma:apriori_1}, we have \(\max_x \varphi_0(x) = 0\), whence \(\varphi_0 \equiv 0\) by Remark~\ref{rem:sign-changing}. Lemma~\ref{lemma:essentialbound} thus leads to a contradiction:
\[
0 =  \lim_{k \to \infty} \left( \textstyle{\frac{\mu_{n_k}}{2}} - \varphi_{n_k}(\frac{P}{2}) \right) \geq \lambda_{K,P} > 0,
\]
which implies that \(\mu({s}) \gtrsim 1\), uniformly for all \({s} \geq 0\).
\end{proof}

Remark \ref{rem:wavespeed bounded away from 1} and Corollary~\ref{cor:essentialbound} show that \(\mu(s)\) is bounded from above and below. A bounded  \(\mu\)  is enough to conclude from \cite[Proposition 4.9]{EK11} that blow-up in \(S\) can only happen by approaching the boundary of \(U\).

\begin{proposition}\label{prop:i implies ii}\cite{EK11}
If alternative (i) in Theorem~\ref{thm:global_whit} occurs, then 
\[
\lim_{s \to \infty}  \left( \frac{\mu(s)}{2} - \max(\varphi(s)) \right) = 0.
\]
In particular, alternative (i) in Theorem~\ref{thm:global_whit} implies alternative  (ii).
\end{proposition}

We also have the following:
\begin{proposition}
In Theorem~\ref{thm:global_whit}, alternative (ii) implies alternative (i).
\end{proposition}
\begin{proof}
Assume for a contradiction that alternative (ii), but not alternative (i), in Theorem~\ref{thm:global_whit} occurs. Then there exists a sequence \((\varphi_n,\mu_n)\) of even solutions to the steady Whitham equation \eqref{eq:steadywhitham} satisfying \(\varphi_n^\prime \geq 0\) on \((-P/2,0)\), \(\varphi_n < \frac{\mu_n}{2}\), and 
\[
\lim_{n\to \infty} \left| \frac{\mu_n}{2} - \varphi_n(0) \right| = 0, 
\]
while \(\varphi_n\) remains bounded in \(C^\alpha(\s)\), \(\alpha > \frac{1}{2}\). Taking a limit along a subsequence  in \(C^{\alpha'}(\s)\), \(\frac12 < \alpha' < \alpha\), we reach a contradiction to 
Lemma~\ref{lemma:essentialbound}, and hence alternative (i) in Theorem~\ref{thm:global_whit} occurs.
\end{proof}

We are now at the final building block for our main result. Pick any sequence \((s_n)_n\) with \(s_n \to \infty\)  as \(n \to \infty\). According to Remark~\ref{rem:wavespeed bounded away from 1} and Corollary~\ref{cor:essentialbound},  \(\mu(s_n)\) is bounded (and bounded away from \(0\) and \(1\)), whence Lemma~\ref{lemma:convergence} implies the existence of a subsequence \((\varphi_{n_k})_k\) converging uniformly to a solution \(\varphi_0\) as \(k \to \infty\). Let  \(\mu_0\) be the wave speed associated to \(\varphi_0\).  By the nodal properties of \(\varphi_{n_k}\), it immediately follows that \(\varphi_0(0) = \frac{\mu_0}{2}\).  Thus, in view of Theorem~\ref{thm:regularity II}, we have proved:

\begin{theorem}\label{thm:main}
In Theorem~\ref{thm:global_whit}, alternatives (i) and (ii) both occur.  Given any unbounded  sequence of positive numbers \(s_n\), there exists a limiting wave obtained as the uniform limit of a subsequence of \((\varphi({s_n}))_n\). The limiting wave solves the steady Whitham equation~\eqref{eq:steadywhitham} with 
\[
\varphi(0) = {\textstyle \frac{\mu}{2}} \quad\text{ and }\quad \varphi \in C^{1/2}(\R). 
\]
It is even,  strictly increasing on \((-P/2,0)\), smooth on \(\R \setminus P\Z\), and has H\"older regularity exactly \(\frac{1}{2}\) at \(x \in P\Z\). 
\end{theorem}

\begin{acknowledgments}
M.E. would like to thank Eugenia Malinnikova for an original idea about proving the positivity of the kernel \(K\), which later led to the proof of \(K\) being convex described in \cite{E15_highest}. M.E. is also thankful to David Lannes and Mathew Johnson for their hospitality and for sharing their expertise  during visits when part of this research was carried out. Both authors are grateful to Mateusz Kwa\'snicki for helpful comments on Stieltjes functions and subordinate Brownian motions. Finally, we thank an unknown referee of this paper who did a very thorough job proof-checking it, and suggested several corrections and improvements.
\end{acknowledgments}

\end{document}